\documentclass[a4paper,twoside,10pt]{article}

\usepackage[a4paper,left=3cm,right=3cm, top=3cm, bottom=3cm]{geometry}
\usepackage[latin1]{inputenc}
\usepackage{mathrsfs}
\usepackage{graphicx}
\usepackage{epstopdf}
\usepackage{amsmath}
\usepackage{appendix}
\usepackage{amsthm}
\usepackage{mathtools}
\usepackage{amssymb}
\usepackage{hyperref}
\usepackage{stmaryrd}
\usepackage{float}
\usepackage{bigints}
\usepackage[noadjust]{cite}
\usepackage{color}
\usepackage[abs]{overpic}
\usepackage[font=footnotesize,labelfont=bf]{caption}
\usepackage{cases}
\usepackage{tikz}
\usepackage{algorithm,algpseudocode}
\usepackage{rotating}
\usepackage{blkarray}
\usetikzlibrary{matrix,calc,arrows}
\usepackage{soul,xcolor}
\usepackage{verbatim}
\usepackage{graphicx}
\usepackage{subcaption}
\usepackage{algorithm}
\usepackage{pifont}
\usepackage{multirow}
\usepackage{bm}
\usepackage[export]{adjustbox}
\usepackage[inline]{enumitem}
\usepackage{tikz}
\usepackage{tikz-cd}
\usepackage[capitalise]{cleveref}
\crefformat{equation}{(#2#1#3)}
\crefrangeformat{equation}{(#3#1#4) to~(#5#2#6)}
\crefmultiformat{equation}{(#2#1#3)}%
{ and~(#2#1#3)}{, (#2#1#3)}{ and~(#2#1#3)}
\usepackage{multirow}

\restylefloat{table}
\theoremstyle{plain}
\newtheorem{theorem}{Theorem}[section]
\newtheorem{corollary}{Corollary}
\newtheorem{lemma}{Lemma}

\theoremstyle{definition}
\newtheorem{definition}{Definition}[section]

\theoremstyle{remark}
\newtheorem{remark}{Remark}

\newcommand{\mc}{\mathcal}
\newcommand{\R}{\mathbb{R}}
\newcommand{\N}{\mathbb{N}}
\newcommand{\Z}{\mathbb{Z}}
\newcommand{\K}{\mc{K}}
\newcommand{\inner}[2]{< #1,#2>}
\newcommand{\cochain}[1]{\mathbf #1}
\newcommand{\torsion}[1]{\bar{#1}^{\,*}}
\newcommand{\rank}{\mathrm{dim}\,}
\newcommand{\dual}[1]{\tilde{#1}}
\newcommand{\abs}[1]{|{#1}|}
\newcommand{\eqclassspan}[1]{\langle{[#1]}\rangle}

\newcommand{\ha}{\mathrm{ha}}
\newcommand{\ho}{\mathrm{ho}}
\newcommand{\co}{\mathrm{co}}

\let\div\relax
\DeclareMathOperator{\grad}{grad}
\DeclareMathOperator{\div}{div}
\DeclareMathOperator{\curl}{curl}

\begin{tiny}
\author{
\normalsize{
}}
\end{tiny}

\date{}

\title{{\textbf{\large{Generators of $H^1(\Gamma, \partial \Gamma^c)$ with $\partial \Gamma^c \subset \partial \Gamma$ for Triangulated Surfaces $\Gamma$: Construction and Classification of Global Loops}}}}
\date{}
\author{{Silvano Pitassi \thanks{IMAG, University of Montpellier, Montpellier, France (silvano.pitassi@umontpellier.fr)}}}

\begin{document}
\maketitle
\begin{abstract}
\noindent
Given a compact surface $\Gamma$ embedded in $\R^3$ with boundary $\partial \Gamma$, our goal is to construct a set of representatives for a basis of the relative cohomology group $H^1(\Gamma, \partial \Gamma^c)$, where $\Gamma^c$ is a specified subset of $\partial \Gamma$.
To achieve this, we propose a novel graph-based algorithm with two key features: it is applicable to non-orientable surfaces, thereby generalizing the construction of Hiptmair and Ostrowski [SIAM J. Comput., 31 (2002)], and it has a worst-case time complexity that is linear in the number of edges of the mesh $\mc{K}$ triangulating $\Gamma$.
Importantly, this algorithm serves as a critical pre-processing step to address the low-frequency breakdown encountered in boundary element discretizations of integral equation formulations.
\medskip

\end{abstract}

\section{Introduction}
\label{sec:intro}
This paper bridges the fields of algebraic topology and computational electromagnetics.
Specifically, it addresses a challenge arising from boundary element method computations for an electromagnetic scattering problem: computing a set of representatives for a basis of the \emph{first relative cohomology group} $H^1(\Gamma, \partial \Gamma^c)$, where $\Gamma$ is a triangulated surface embedded in $\R^3$ and $\partial \Gamma^c \subset \partial \Gamma$.
To tackle this, we propose a novel graph-based algorithm with a worst-case linear-time complexity.
\paragraph*{Background and motivation.}
Throughout this paper, we consider the following prototypical electromagnetic scattering problem, which serves as the primary motivation for the algorithm discussed here.
Specifically, an incident harmonic electric field $\bm{E}^i$, propagating in a space with permittivity $\epsilon$ and permeability $\mu$, impinges upon a triangulated surface $\Gamma$ without boundary, which is a perfect conductor.
This results in a scattered electric field $\bm{E}^s$. To determine $\bm{E}^s$, we solve the \emph{Electric Field Integral Equation (EFIE)}
\[
\gamma_{\mathrm{T}} \circ \mathfrak{U}_k(\bm{J}) = -\gamma_{\mathrm{T}} \bm{E}^i,
\]
which involves finding the induced surface current density $\bm{J}$ on $\Gamma$ determining $\bm{E}^s$ via the \emph{EFIE integral operator} $\gamma_{\mathrm{T}} \circ \mathfrak{U}_k$ \cite{buffa2003electric}.
Here, $\gamma_{\mathrm{T}}$ denotes the tangential trace operator on $\Gamma$.

Crucially, when the EFIE is discretized using the \emph{Boundary Element Method (BEM)}, two numerical challenges arise as the frequency of the involved harmonic fields approaches zero---a well-documented phenomenon referred to as the \emph{low-frequency breakdown}.
The first challenge is the growth in the condition number of the resulting linear system after BEM discretization, while the second is the occurrence of round-off errors due to finite precision in the numerical integration of the EFIE's source term.
These two effects can lead to inefficiencies in standard numerical solvers, requiring significantly higher computational resources (e.g., more memory or more solver iterations), or even causing the solution to fail to converge \cite{adrian2021electromagnetic}.

To address the low-frequency breakdown of the EFIE, the key idea is to decompose the current density $\bm{J}$ as the sum $\bm{J} = \bm{J}_{\Lambda} + \bm{J}_{\Sigma}$, where $\bm{J}_{\Lambda}$ is \emph{solenoidal} (i.e., $\div_\Gamma \bm{J}_{\Lambda} = 0$, with $\div_\Gamma$ denoting the surface divergence operator on $\Gamma$) and $\bm{J}_{\Sigma}$ is non-solenoidal.
This decomposition is essentially derived from the Hodge decomposition theorem applied to the $L^2$-based \emph{de Rham complex} on $\Gamma$, along with the definition of the first (absolute) cohomology group $H^1(\Gamma)$ as a quotient space, so that we can write:
\begin{equation}
\bm{J} = \underbrace{\curl_\Gamma \bm{\psi} + \bm{g}}_{\eqqcolon \bm{J}_{\Lambda}} + \underbrace{\grad_\Gamma \varphi}_{\eqqcolon \bm{J}_{\Sigma}},
\label{eq:quasi-Helmholtz}
\end{equation}
where $\grad_\Gamma$ is the $L^2$-adjoint of $\div_\Gamma$, and $\bm{g}$ represents a cohomology class in $H^1(\Gamma)$.
The vector field $\bm{\psi}$ is referred to as the \emph{stream function} or \emph{local loop}, and $\bm{g}$ is called a \emph{(cohomology) generator} or \emph{global loop}.

Two approaches for computing this decomposition have been proposed in the engineering literature.
Initially, the \emph{loop-star} decomposition was introduced \cite{vecchi1999loop, andriulli2012loop}. This approach requires computing a set of generators for the image spaces $\mathrm{Im} \, \curl_\Gamma$ and $\mathrm{Im} \, \grad_\Gamma$, as well as a set of generators for the cohomology space $H^1(\Gamma)$.
More recently, \emph{quasi-Helmholtz projectors} have been proposed as an alternative to avoid explicitly computing cohomology generators \cite{andriulli2012well}.
This method computes the non-solenoidal part $\bm{J}_{\Sigma}$ via an orthogonal projection $P_{\Sigma}$ onto $\mathrm{Im} \, \grad_\Gamma$, while the solenoidal part $\bm{J}_{\Lambda}$ is found using the complementary orthogonal projection $P_{\Lambda}$.

However, we emphasize that the computation of generators of $H^1(\Gamma)$ is, in practice, strictly necessary to address the low-frequency breakdown of the EFIE, even when quasi-Helmholtz projectors are used.
This point is supported by recent work \cite{hofmann2023low}, which introduces a technique for reducing round-off errors in the numerical evaluation of the EFIE's source term.
Notably, the methodology outlined in Sections IV A and B of that paper clearly shows the need for a set of generators of $H^1(\Gamma)$, as they appear explicitly in the definition of the proposed quasi-Helmholtz projectors (cf. \cite[Equations (24)--(28)]{hofmann2023low}).
\paragraph*{Review of cohomology computations in the numerical analysis community.}
We begin by pointing out early works by Hiptmair and Ostrowski in \cite{hiptmair2002generators} (which we refer to as the \emph{HO-algorithm} hereafter), along with a non-exhaustive list of related works in \cite{Eppstein2002DynamicGO, lazarus2001computing, dey2013efficient}.
These articles focus on a surface $\Gamma$ without boundary and aim to compute generators of the first homology group $H_1(\Gamma)$.
However, in numerical analysis and engineering applications, what is typically needed are generators of $H^1(\Gamma)$, rather than those of $H_1(\Gamma)$.

Since $\Gamma$ is orientable, it is possible to recover generators of $H^1(\Gamma)$ from those of $H_1(\Gamma)$ using the Poincar\'e-Lefschetz Duality Theorem \cite{gross2004electromagnetic}.
This approach is mentioned in the introduction of \cite{hiptmair2002generators}, though the actual implementation details are missing.
Instead, an algorithm that implements this construction was first developed by D\l{}otko in \cite{dlotko2012fast} and later reorganized by D\l{}otko and Specogna in the series of works \cite{dlotko2013physics, dlotko2014lazy} (referred to as the \emph{DS-algorithm}).

Now, consider the case where $\Gamma$ is a surface with a non-empty boundary $\partial \Gamma$.
In this case, note that the DS-algorithm can still be applied with a simple modification of the construction, as described in \cite[Section 6]{dlotko2012fast}.
However, in the context of our electromagnetic application, the cohomology space $H^1(\Gamma)$ is no longer appropriate.
We must instead account for the boundary condition $\bm J \cdot \bm n_{\partial \Gamma} = 0$ on $\partial \Gamma$, which arises from the EFIE model.

To satisfy this boundary condition, it can be shown that the vector field $\bm g$ in \cref{eq:quasi-Helmholtz} must belong to the first \emph{relative cohomology space} $H^1(\Gamma, \partial \Gamma)$ \cite{gross2004electromagnetic}.
Therefore, algorithms like the DS-algorithm or the HO-algorithm are unsuitable for computing generators of $H^1(\Gamma, \partial \Gamma)$, since $H^1(\Gamma)$ and $H^1(\Gamma, \partial \Gamma)$ are generally quite different spaces;
see \cref{fig:disk}.

As far as we are aware, no combinatorial algorithm has been devised to address this case.
\begin{figure}
\begin{center}
\includegraphics[scale=0.3]{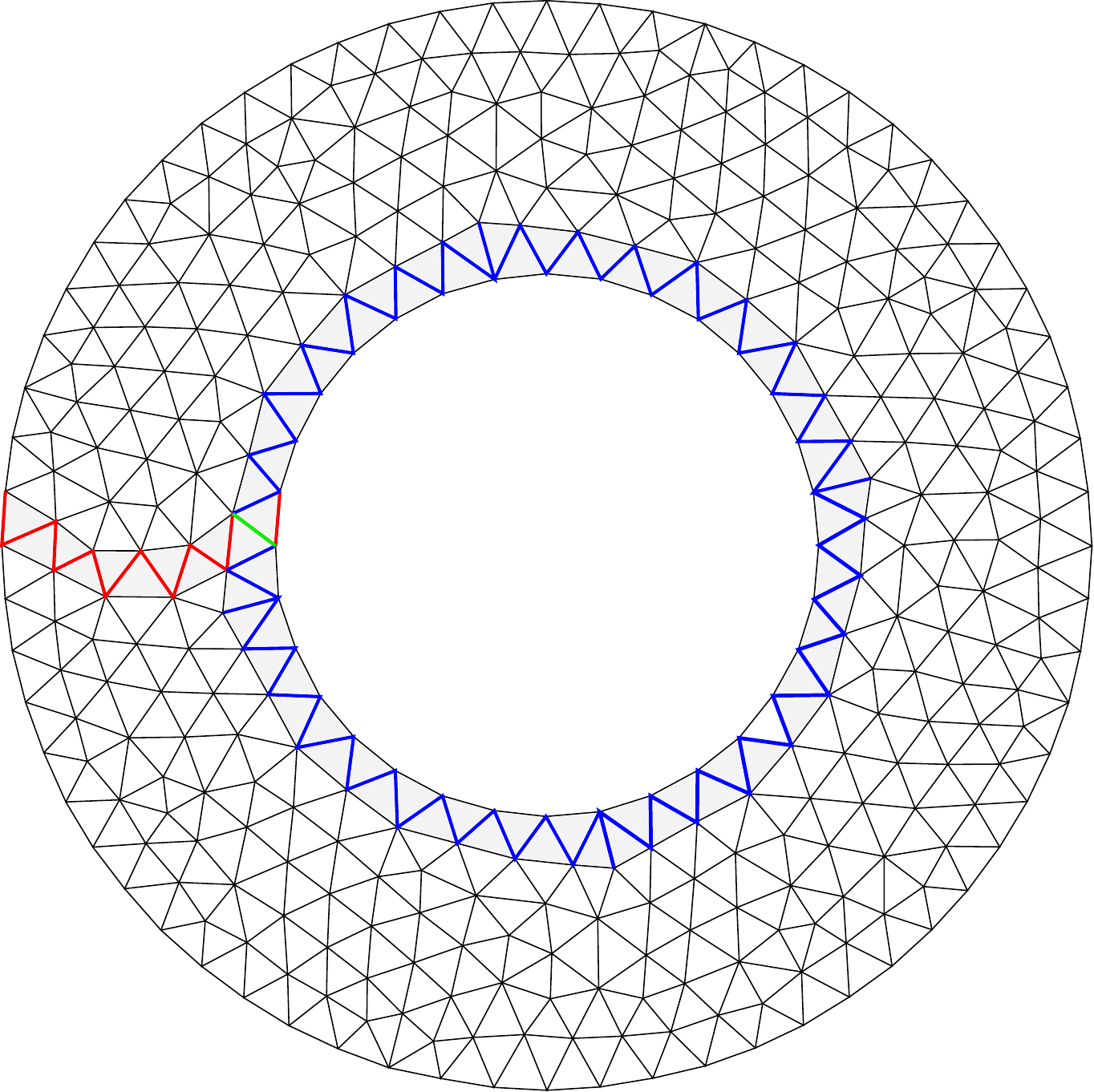}
\end{center}
\caption{A triangulation of an annulus $\Gamma$.
The support of the representative of the $H^1(\Gamma)$ generator, as formally defined in \cref{sec:homology}, is shown in red, while the support of the representative of the $H^1(\Gamma, \partial \Gamma)$ generator is shown in blue.
The green part shows the intersection of the two supports.
The shaded grey triangles indicate the support of the corresponding generators as functions in $L^2(\Gamma)$, obtained by reconstruction using the standard Raviart-Thomas basis functions.}
\label{fig:disk}
\end{figure}

In this paper, we build upon this observation and consider the more general problem of computing generators of $H^1(\Gamma, \partial \Gamma^c)$, where $\partial \Gamma^c$ is a suitable subset of the boundary $\partial \Gamma$ of $\Gamma$.
This general version arises in the EFIE model, specifically for cases where the surface $\Gamma$ is connected to external lumped-element circuits through suitable \emph{contacts} or \emph{ports}, such as in antennas and their connections with electronic circuits.

From a physical perspective, these connections result in net currents flowing through the body from one electrode to another, a phenomenon that would otherwise be absent.
Crucially, this modifies the topology of the surface. For example, in the simplest case of a simply connected surface $\Gamma$, these connections make the surface non-simply connected.
We discuss the details of the EFIE model with contact regions in \cref{sec:EFIE}.

\paragraph*{The main results and their novelty w.r.t.
literature.}
The primary objective of this paper is to provide an efficient algorithm for constructing and classifying the generators of the first relative cohomology group $ H^1(\Gamma, \partial \Gamma^c) $, where $ \Gamma $ is an arbitrary triangulated surface embedded in $ \R^3 $ and $ \partial \Gamma^c $ is a subset of its boundary.
Specifically, we present two key contributions: one for the mathematical community and the other for the engineering community.

On the mathematical side, a significant novel aspect is that our algorithm is provably correct for non-orientable surfaces, such as the \emph{M\"obius strip}.
Since we assume that $ \Gamma $ is embedded in $ \R^3 $, it follows that $ \Gamma $ must have a non-empty boundary. This is because a non-orientable surface without boundary cannot be embedded in $ \R^3 $ (e.g., the Klein bottle; see \cite[Theorem~27.11]{greenberg1981algebraic} for the general statement of this fact), while every surface with boundary can be embedded in $ \R^3 $. The key challenge for non-orientable surfaces is that a generator of the \emph{torsion subgroup} 
appears when considering relative homology with integer coefficients (a fact evident from a canonical representation of the M\"obius strip).
The problem, then, is how to distinguish the torsion generator from other homology generators.

Crucially, in the proof of correctness presented in \cref{sec:proof}, we explicitly identify the torsion generator.
In doing so, we also resolve an open problem mentioned in \cite[Section 7]{dlotko2012fast}.
The proof relies on concepts from \emph{Discrete Morse Theory}, which has several far-reaching implications.
First, it generalizes the construction used in the HO-algorithm, making our algorithm a natural extension of the HO-algorithm to the case of non-orientable surfaces.
Second, it provides an alternative proof of the classical Classification Theorem for Compact Surfaces \cite{gallier2013guide}, which follows as a corollary of our specific constructions.

Regarding the engineering contribution, the first aspect is a rigorous definition of \emph{global loops} \cite{wilton1981improving} within the framework of \emph{cohomology theory}.
In \cref{def:global.loops}, at the end of \cref{sec:EFIE}, we show that when global loops are defined as generators of the first relative cohomology group $H^1(\Gamma, \partial \Gamma^c)$ of a triangulated surface that excludes contact regions, the decomposition of the current density $\bm{J}$ in \cref{eq:quasi-Helmholtz} holds.
This formal definition has several important implications. First, it covers cases in which $\Gamma$ has no contact regions and no boundary.
In these situations, the space $H^1(\Gamma, \partial \Gamma^c)$ reduces to both $H^1(\Gamma, \partial \Gamma)$ and $H^1(\Gamma)$, since we have $\partial \Gamma^c = \partial \Gamma$ and $\partial \Gamma = \emptyset$, respectively.
Thus, our definition naturally encompasses a broad range of modeling contexts, correctly identifying the appropriate cohomology space, relative or absolute, in each case.

Second, our definition is based directly on the \emph{cochain complex} (see \cite[Section 2.6]{arnold2018finite}) supported on a triangulation $\mathcal K$ of $\Gamma$, rather than on the standard $L^2$ de~Rham complex typically used in the engineering literature (see \cite[Section IV]{adrian2021electromagnetic}).
This perspective provides new insights into algorithmic construction, enabling implementations based on efficient graph-theoretic techniques.
In particular, our algorithm has a worst-case time complexity that is linear in the number of edges of the triangulation, making it highly efficient for engineering applications such as those involving the EFIE model.
It is worth noting that, in the context of engineering computations of global loops, it has even been suggested that ``it is impossible to provide a complete description of this family of functions'' \cite{andriulli2012well}.

Another important aspect is the treatment of contact regions. While the EFIE model in the presence of contacts has been studied in engineering works \cite{miano2005surface, wang2004generalized}, the focus has primarily been on the design of the coupling between the EFIE model and external circuit equations, as well as the quasi-static assumptions near contact regions.
Mathematically, contributions have been made to establish links with relative cohomology, beginning with seminal ideas in \cite{gross2004electromagnetic} and revisited in works such as \cite{suuriniemi2004state, suuriniemi2007driving, hiptmair2021electromagnetic}, among others.
However, these contributions have largely focused on orientable domains, where the Poincar\'e-Lefschetz Duality Theorem can be applied.

In contrast, when a torsion generator is present, an additional complication arises, particularly the introduction of an extra generator associated with contact regions.
This stands in stark contrast to the orientable case, where a clear physical interpretation is also available.
Indeed, techniques that rely on duality theorems do not necessarily apply in non-orientable cases, as the hypotheses underlying these theorems are not satisfied.

We believe that the combination of all these features opens up new possibilities for efficient numerical procedures to solve the EFIE.
\section{Preliminaries from algebraic topology}
\label{sec:math}

In this preliminary section, we introduce concepts from algebraic topology and provide specific definitions tailored to our purposes.
In particular, we define the concept of a cochain as a model for a discrete vector field.

In this paper, we consider two types of cochains. The first are integer-valued cochains, such as representatives of the generators of the first cohomology group over the integers.
The second are real-valued cochains, which are used to model physical variables.
\subsection{CW complex and triangulations}

A \emph{$2$-manifold $\Gamma$ with boundary} is a compact Hausdorff topological space in which every point has a neighborhood homeomorphic to a relatively open subset of either $\R^2$ or the half-plane $\{(a, b) \in \R^2 \mid a \geq 0\}$.
The \emph{boundary} of $\Gamma$, denoted by $\partial \Gamma$, consists of all points with neighborhoods homeomorphic to a relatively open subset of the half-plane.
We say that $\Gamma$ is a \emph{$2$-manifold without boundary}, or simply a $2$-manifold, if $\partial \Gamma$ is empty.

An (\emph{open}) $m$-cell, for $ m \in \{0, 1, 2\} $, is a topological space homeomorphic to the interior $\mathrm{int}\,B^m$ of an $m$-ball $B^m$ (note that $ B^0 $ is a single point).
Let us now introduce the concept of a CW complex $\mc{K}$ defined on $\Gamma$.
\begin{definition}[CW complex]
\label{def:CW_complex}
(See \cite[Section~4]{whitehead1949combinatorial}.)\par
A ($2$-dimensional) \emph{CW complex} $\mc{K}$ on $\Gamma$ is a partition of $\Gamma$ into $m$-cells $x$ such that the following holds:  
for each $m$-cell $x$, with corresponding closure $\mathrm{cl}\,x$, there exists a continuous surjection (the \emph{gluing map}) $ g_x : B^m \to \mathrm{cl}\,x $, satisfying:  
\begin{enumerate}[label=(\roman*)]
    \item The restriction of $ g_x $ to the interior, $ g_x: \mathrm{int}\,B^m \to x $, is a homeomorphism.
\item The image of the boundary, $ g_x(\partial B^m) $, is contained in the union of a finite number of cells of dimension $k < m$.
\end{enumerate}
\end{definition}
For $ m \in \{0, 1, 2\} $, we denote by $\mc{K}^{(m)}$ the set of all $m$-cells of $\mc{K}$.

A CW complex $\K$ is called \emph{regular} if, for every $m$-cell $x$, the restriction of the gluing map $ g_x : \partial B^m \to g_x(\partial B^m) $ is a homeomorphism.

A \emph{regular} CW complex $\mc{K}$ where each $m$-cell is an $m$-simplex is also referred to as a \emph{triangulation}.
In this case, we adopt a more familiar notation and terminology for the $m$-cells.
Specifically: a \emph{node} $ n $ is a $0$-cell in $ V \coloneqq \mc{K}^{(0)} $, an \emph{edge} $ e $ is a $1$-cell in $ E \coloneqq \mc{K}^{(1)} $, and a \emph{face} $ f $ is a $2$-cell in $ F \coloneqq \mc{K}^{(2)} $.
Accordingly, we denote a triangulation $\mc{K}$ by $ \mc{K} = (V, E, F) $.

A \emph{CW subcomplex} $\mc{S}$ of $\mc{K}$ is a union of $m$-cells $ x \in \mc{K} $ such that, if $ x \subseteq \mc{S} $, then the closure of $ x $ is also a subset of $\mc{S}$.

If $\Gamma$ has a non-empty boundary $\partial \Gamma$, we denote by $\partial \mc{K}$ the subcomplex of $\mc{K}$ induced on $\partial \Gamma$ (note that $\partial \Gamma$ is itself a manifold because $\Gamma$ is).

The \emph{$1$-skeleton} of $\mc{K}$ is the union $\mc{K}^{(0)} \cup \mc{K}^{(1)}$, i.e., the set of all $0$- and $1$-cells.
It is straightforward to observe that the $1$-skeleton of $\mc{K}$ can be regarded as a graph whose vertices and edges are $\mc{K}^{(0)}$ and $\mc{K}^{(1)}$, respectively.

Using this identification, a \emph{tree} $T = (V_T, E_T)$ of $\mc{K}$ refers to a tree of the $1$-skeleton of $\mc{K}$, viewed as a graph.
We then say that $T$ is a \emph{spanning tree} if the set $V_T$ includes all the $0$-cells of $\mc{K}$, that is, if $V_T = \mc{K}^{(0)}$.
\subsection{Chains and cochains}
\label{sec:homology}

Let $\K$ be a CW complex, and let $G$ be either the commutative ring of integers $\Z$ or reals $\R$.
We start by fixing an orientation of all its $m$-cells (see, e.g., \cite[Chapter 2]{cooke2015homology}).

The $m$-th \emph{chain group} $C_m(\K;G)$ is the group of all formal sums $c = \sum_{x \in \K^{(m)}} a_x x$ with $a_x \in G$.
We call each element of $C_m(\K;G)$ a \emph{$m$-chain}.
For two $m$-chains $c=\sum_{x \in \K^{(m)}} a_x x$ and $c' = \sum_{x \in \K^{(m)}} b_{x} x$, we define their \emph{scalar product} by
\[
    \inner{c}{c'} \coloneqq \sum_{x \in \K^{(m)}} a_x b_x.
\]

The $m$-th \emph{cochain group} $C^m(\K;G)$ is the group of all (linear) maps $\cochain c : C_m(\K;G) \to G$ such that $\cochain c(c) = \cochain c (\sum_{x \in \K^{(m)}} a_x x) = \sum_{x \in \K^{(m)}} a_x \cochain c(x)$ for every $c \in C_m(\K;G)$.
We call each element of $C^m(\K;G)$ a \emph{$m$-cochain}.

For a $m$-cochain $\cochain c \in C^m(\K;G)$, its \emph{support} is the set of all $m$-cells $x \in \K^{(m)}$ such that $\cochain c(x) \neq 0$.

Notice that to compute the value $\cochain c(c)$ it is sufficient to know the values $\cochain c(x)$ on every $m$-cell.
Therefore, we can associate to the $m$-cochain $\cochain c$ the $m$-chain $\phi(\cochain c) \coloneqq \sum_{x \in \K^{(m)}} a_x x$ with $a_x = \cochain c(x)$ for $x \in \K^{(m)}$ in such a way that, for each $m$-chain $c' \in C_m(\K;G)$, it holds
$
    \cochain c(c') = \inner{\phi(\cochain c)}{c'}
$
In what follows, we will simply write $\inner{\cochain c}{c'}$ in place of $\inner{\phi(\cochain c)}{c'}$ so that, by the above equality, this expression denotes also the value $ \cochain c(c')$.

For $m \in \{1,2\}$, let $\iota : \K^{(m)} \times \K^{(m-1)} \to \Z$ be the \emph{incidence number} between a $m$-cell and a $(m-1)$-cell (see \cite[Section 2.2]{hatcher}, and, in particular, the cellular boundary formula on p.\ 140; equivalently, see \cite{cooke2015homology}).
If $\iota(x, x') \neq 0$, then we say that $x$ is \emph{incident} on $x'$, or simply that $x$ and $x'$ are incident.

We use incidence numbers to define the \emph{boundary map} $\partial_m : C_m(\K;G) \to C_{m-1}(\K;G)$ for $m \in \{1,2\}$.
For each $m$-cell $x \in \K^{(m)}$ we define $\partial_m x \coloneqq \sum_{x' \in \K^{(m-1)}} \iota(x,x') x'$, and for a $m$-chain $c = \sum_{x \in \K^{(m)}} a_x x \in C_m(\K;G)$, by linearity, we define $\partial_m c \coloneqq \sum_{x \in \K^{(m)}} a_x \partial_m x$.
It can be checked that $\partial_{m+1} \circ \partial_{m} = 0$ for $m \in \{0,1\}$.

The \emph{coboundary map} $\delta^m : C^{m-1}(\K;G) \to C^{m}(\K;G)$, for $m \in \{0,1\}$, is defined as follows.
For a $m$-cochain $\cochain c \in C^{m}(\K;G)$, $\delta^m \cochain c$ is the (unique) $(m+1)$-cochain such that
\[
\inner{\delta^m \cochain c}{ c'} = \inner{\cochain c}{\partial_{m+1} c'}
\]
for all $c' \in C_{m+1}(\K;G)$.
Using $\partial_{m+1} \circ \partial_{m} = 0$, it can be checked that $\delta^m \circ \delta^{m+1}$ for $m \in \{0,1\}$, so that the (graded) family $(C^m(\K;G), \delta^m)$ with $m \in \{0,1,2\}$ is a \emph{(cochain) complex} once we also set $\delta^2 \coloneqq 0$ \cite{arnold2018finite}.
\subsection{Homology and cohomology}
The boundary operator allows us to classify $m$-chains in $\K$ into key subgroups.

We define the subgroup of \emph{$m$-cycles} as the kernel of the boundary operator, $Z_m(\K;G) \coloneqq \ker \partial_m = \{c \in C_m(\K;G) \mid \partial_m c = 0\}$, and the subgroup $B_m(\K;G) \coloneqq \mathrm{im}\, \partial_{m+1} = \{\partial_{m+1} c \mid c \in C_{m+1}(\K;G) \}$, which consists of elements that are the image of the $(m+1)$-th boundary operator.
An essential property is that $B_m(\K;G)$ is a subset of $Z_m(\K;G)$.

The $m$-th \emph{homology group} is then the quotient group $H_m(\K;G) \coloneqq Z_m(\K;G)/B_m(\K;G)$.

When the coefficient ring is the integers $G = \mathbb{Z}$, the homology group $H_m(\K; \mathbb{Z})$ is a finitely generated Abelian group.
By the Fundamental Theorem for Finitely Generated Abelian Groups (see e.g. \cite[Theorem~3.61]{kaczynski2006computational}), it is isomorphic to the direct sum of a free Abelian group and a \emph{torsion subgroup}.

In a similar manner, we define the subgroup of \emph{$m$-cocycles} $Z^m(\K;G) \coloneqq \ker \delta^m = \{\cochain c \in C^m(\K;G) \mid \delta^m \cochain c = 0\}$, and the subgroup $B^m(\K;G) \coloneqq \mathrm{im}\, \delta^{m-1} = \{\delta^{m-1} \cochain c' \mid \cochain c' \in C^{m-1}(\K;G)\}$.
The $m$-th \emph{cohomology group} is the quotient group $H^m(\K;G) \coloneqq Z^m(\K;G) / B^m(\K;G)$.

In the following, we will also need the concept of \emph{relative cohomology}, which focuses on a complex relative to a subcomplex.
Specifically, we will focus on the CW subcomplex $\partial \K^c$ induced by $\K$ on a closed subset $\partial \Gamma^c$ of $\partial \Gamma$.

The group of $m$-th \emph{relative cochains}, denoted $C^m(\K,\partial \K^c;G)$, is defined as the subgroup of $C^m(\K;G)$ consisting of $m$-cochains that vanish on all $m$-cells of $\partial \K^c$.
This means that a $m$-cochain $\cochain c \in C^m(\K;G)$ is a relative $m$-cochain if its support does not contain any $m$-cells in $\partial \K^c$.

The coboundary operator $\delta^m$ restricts to a map on these groups, allowing the definition of the $m$-th \emph{relative cohomology group} $H^m(\K,\partial \K^c;G)$ analogously to the absolute case.

When the coefficient ring is the reals, $G = \mathbb{R}$, the $m$-th cohomology group $H^m(\K; \mathbb{R})$ is a real vector space, and its dimension is denoted by $\beta_m \in \mathbb{N}$.
A basis for $H^m(\K; \mathbb{R})$ is a set of equivalence classes $\{[\cochain{g}_1], \dots, [\cochain{g}_{\beta_m}]\} \subset H^m(\K;\mathbb{R})$ such that every element in $H^m(\K;\mathbb{R})$ can be uniquely written as a linear combination $\sum_{i=1}^{\beta_m} \alpha_i[\cochain g_i]$ with $\alpha_i \in \mathbb{R}$.
The set of $m$-cocycles $\mc G \coloneqq \{\cochain{g}_1, \dots, \cochain{g}_{\beta_m}\}$ is then called a set of \emph{cohomology generators}.
These considerations and terminologies also apply to the relative cohomology group $H^m(\K, \partial \K^c; \mathbb{R})$.
\begin{remark}[Relationship between cohomology with coefficients in $\mathbb{Z}$ and $\mathbb{R}$]
The relationship between cohomology with integer coefficients $H^m(\K; \mathbb{Z})$ and real coefficients $H^m(\K; \mathbb{R})$ is crucial for our application.
For the numerical approximation of physical fields, we use generators of $H^m(\K; \mathbb{R})$, while computational topology algorithms typically provide generators of $H^m(\K; \mathbb{Z})$.

Fortunately, generators for $H^m(\K; \mathbb{R})$ can be derived from those for $H^m(\K; \mathbb{Z})$.
The basis for the real vector space $H^m(\K; \mathbb{R})$ corresponds to the free part of the abelian group $H^m(\K; \mathbb{Z})$, effectively disregarding its torsion subgroup.
This connection is formalized by the Universal Coefficient Theorem; see \cref{sec:characterization}, and in particular \cref{cor:characterization}, for its implications in our specific constructions.
\end{remark}

\begin{remark}[Notation for cohomology with coefficients in $\mathbb{R}$]
Unless otherwise specified, cohomology groups will be taken with real coefficients.
In particular, for the remainder of the paper we write $H^m(\K)=H^m(\K;\R)$, while we continue to write $H^m(\K;\Z)$ when integer coefficients are intended.
The same convention applies to relative cohomology, e.g. $H^m(\K,\partial\K^c)=H^m(\K,\partial\K^c;\R)$.
\end{remark}

\section{Construction and classification of generators of $H^1(\Gamma, \partial \Gamma^c)$}
\label{sec:algorithms}

Let $\K= (V, E, F)$ be a triangulation covering $\Gamma$.
For a specified subset $\partial \Gamma^c$ of the boundary $\partial \Gamma$, let $\partial \K^c$ denote the subtriangulation induced by $\K$ on $\partial \Gamma^c$.

In this section, we present a general algorithm to compute a specific set $\mc G = \{\cochain g_i \}_{i=1}^{\beta_1}$ of generators of $H^1(\K, \partial \K^c)$.
The algorithm is detailed in \cref{alg:global.loops}.
The key idea is that the set $\mc G$ can be partitioned into three disjoint subsets, each corresponding to distinct topological features of the pair $(\K, \partial \K^c)$.
We use the following mnemonic notation for these classes:
\begin{itemize}
    \item $\mc G^\ha$: generators associated with the genus of $\K$ (``handles'').
\item $\mc G^\ho$: generators associated with the connected components of $\partial \K$ (``holes'').
\item $\mc G^\co$: generators associated with the connected components of the complementary set $\partial \K \setminus \partial \K^c$.
\end{itemize}
These subsets may be empty depending on the specific triangulation $\K$, and they can be computed independently.

We assume that $\mc K$ is connected for the purposes of this discussion. This assumption is not restrictive;
indeed, if $\K$ consists of $p$ connected components $\K_i$, we have the following decomposition
\[
H^1(\K, \partial \K^c) \cong \bigoplus_{i=1}^p H^1(\K_i, \partial \K_i^c),
\]
which implies that it suffices to consider each connected component $\K_i$ individually.

\subsection{Intepreting cocyles as cycles on the dual graph $\dual G$}
In the constructions used in \cref{alg:global.loops}, we require the definition of a special graph associated with the triangulation $\K=(V,E,F)$.
To define it, we introduce the sets of dual nodes $\dual V$ and dual edges $\dual E$ via a bijective \emph{duality map} $D: E \cup F \to \dual V \cup \dual E$ and a \emph{boundary map} $D_\partial: E_\partial \to \dual V_\partial$.
These are constructed as follows:
\begin{itemize}
    \item For each face $f \in F$, the corresponding \emph{dual node} is $\dual n = D(f) \in \dual V$.
\item For each internal edge $e \in E \setminus E_\partial$, shared by two faces $f_e^{(1)}, f_e^{(2)} \in F$, the corresponding \emph{dual edge} is $\dual e = \{D(f_e^{(1)}), D(f_e^{(2)})\} = D(e) \in \dual E$.
    \item For each boundary edge $e \in E_\partial$, incident to the unique face $f_e \in F$, we define a (boundary) dual node $\dual n = D_\partial(e) \in \dual V_\partial$ and a corresponding dual edge $\dual e = \{D(f_e), D_\partial(e)\} = D(e) \in \dual E$.
\end{itemize}
Then, we define the graph $\dual G = (\dual V, \dual E)$ to be the \emph{dual graph} of $\K$.

A \emph{path} $\dual p$ of $\dual G$ is a sequence of $l \in \N, l \geq 1$ distinct dual nodes $(\dual n_1, \dots, \dual n_l)$ and $l-1$ distinct dual edges $(\dual e_1, \dots, \dual e_{l-1})$ such that $\dual e_i=\{\dual n_{i}, \dual n_{i+1}\}$ for $i \in \{1, \dots, l-1\}$ ($l=1$ means the path reduces to the single dual node $\dual n_1$).
By duality via the map $D$, there is a corresponding sequence of distinct $l$-faces  $(f_1=D^{-1}(\dual n_1),\dots, f_l = D^{-1}(\dual n_l))$ and $l-1$ distinct edges $(e_1=D^{-1}(\dual e_1), \dots, e_{l-1} = D^{-1}(\dual e_{l-1}))$ of $\K$.

We use paths on the dual graph $\dual G$ to construct $1$-cochains of $\K$ according to the following procedure.
\begin{algorithm}[H]
\caption{Construct a $1$-cocycle $\cochain{g}$ on $\K$ from a path 
$\dual p$ and a pair of edges $(e,e')$ with $e\subset f_1$ and $e'\subset f_l$}
\begin{algorithmic}[1]
    \State initialize $\cochain g = \cochain 0$
    \State set $\inner{\cochain g}{e}$ to 1
    \For {$i = 1$ to $l-1$} 
        \State let $f_i = D^{-1}(\dual n_i)$ and $e_i = D^{-1}(\dual e_i)$ 
        \State compute
        \[   
     \inner{\cochain g}{e_i}
        \coloneqq 
        -\frac{\iota(f_i, e_i)}{\iota(f_i, e_{i-1})}
        \inner{\cochain g}{e_{i-1}}
        \]
        \Comment{with the conventions $e_0 := e$ and $e_l := e'$}
    \EndFor
    \If {$\inner{\cochain g}{e'} = 0$}
        \State \Return $\cochain 0$
    \Else
        \State \Return $\cochain g$
   
  \EndIf
\end{algorithmic}
\label{alg:cocycle}
\end{algorithm}

\subsection{The first class of generators: the set $\mc G^\ha$}
The first class of generators, denoted by $\mc G^\ha$, corresponds to the non-trivial genus of the surface.
These generators are the most difficult to identify combinatorially.
Their construction relies on a pair of trees, one defined on $\K$ and the other on its dual graph $\tilde G$.
\begin{algorithm}[H]
\caption{Construct the set $\mc G^\ha$}
\begin{algorithmic}[1]
    \State construct a spanning tree $T$ on each connected component of $\partial \K$
    \State extend $T$ to a spanning tree $(V_T, E_T)$ on the whole $\K$
    \State construct a spanning tree $\dual T = (\dual V_{\dual T}, \dual E_{\dual T})$ on the subgraph $\dual G' = (\dual V, \dual E \setminus \bigcup_{e \in E_T} \{D(e)\})$ \Comment{i.e., the subgraph of $\dual G$ made of all dual edges in $\dual G$ that are not dual to an edge in $T$}
    \For{each connected component $\partial \K_k$ 
of $\partial \K$}
        \State select the unique edge $e$ of $\partial \K_k$ not contained in $E_T$ 
        \Comment{the uniqueness follows from the fact that $T$ is constructed first on the boundary, as in line 1}
        \State add the dual edge $D(e)$ to $\dual E_{\dual T}$ and the dual node $D_\partial(e)$ to $\dual V_{\dual T}$
    \EndFor
    \State let $E_M := \{e \in E \mid e \notin E_T, D(e) \notin \dual E_{\dual T}\}$ \Comment{i.e., the 
set of edges not in $T$ whose dual is not in $\dual T$}
    \For{each $e \in E_M$}
        \State let $f_{e}^{(1)}, f_{e}^{(2)}$ be the unique two faces of $\K$ incident on $e$
        \State construct the unique path $\dual p_{e}$ in $\dual T$ from $D(f_{e}^{(1)})$ to $D(f_{e}^{(2)})$
        \State apply \cref{alg:cocycle} with input $\dual p_e$ and the pair $(e,e)$ to get $\cochain g_e^{\ha}$

    \EndFor
    \State let $E_M^{\rm{II}} := \{ e \in E_M \mid \cochain g_e^{\ha} = \cochain 0 \}$ \Comment{i.e., the set of edges for which the $1$-cochain $\cochain g_e^{\ha}$ computed from \cref{alg:cocycle} is the zero $1$-cochain}
    \If{ $\abs{E_M^{\rm{II}}} = 0$}
        \State \textbf{return} $\mc G^{\ha} \coloneqq \{\cochain g^{\ha}_1, \dots, \cochain g^{\ha}_{N^{\ha}}\} = \{\cochain g_e^{\ha}\}_{e \in E_M}$ \Comment{$N^{\ha}$ is equal to $\abs{E_M}$}
    \Else \Comment{i.e., if $\abs{E_M^{\rm{II}}} > 0$}
        \State select an edge $e^* \in E_M^{\rm{II}}$
    
     \For{each $e \in E_M^{\rm{II}} \setminus \{e^*\}$}
            \State let $f_{e}^{(1)}, f_{e}^{(2)}$ be the unique two faces of $\K$ incident on $e$
            \State let $f_{e^*}^{(1)}, f_{e^*}^{(2)}$ be the unique two faces of $\K$ incident on $e^*$
            \State construct the unique path $\dual p_{e,e^*}^{(1)}$ in $\dual T$ from $D(f_{e}^{(1)})$ to $D(f_{e^*}^{(1)})$
            \State construct the unique 
path $\dual p_{e,e^*}^{(2)}$ in $\dual T$ from $D(f_{e}^{(2)})$ to $D(f_{e^*}^{(2)})$
            \State apply \cref{alg:cocycle} with input $\dual p_{e,e^*}^{(1)}$ and $(e,e^*)$ to get $\cochain g_{e}^{(1)}$
            \State apply \cref{alg:cocycle} with input $\dual p_{e,e^*}^{(2)}$ and $(e,e^*)$ to get $\cochain g_{e}^{(2)}$
            \State define $\cochain g_e^{\ha}$ by
            \begin{align*}               
  \inner{\cochain g_e^{\ha}}{e'} \coloneqq
                    \begin{cases}
                        \inner{\cochain g_e^{(1)}}{e'}  &\text{if } e' \in \{e,e^*\},\\
                        \inner{\cochain g_e^{(1)}}{e'} + \inner{\cochain g_e^{(2)}}{e'}  &\text{if } e' \notin \{e, e^*\}           
          \end{cases}
            \end{align*}
        \EndFor
        \State \textbf{return} $\mc G^{\ha} \coloneqq \{\cochain g^{\ha}_1, \dots, \cochain g^{\ha}_{N^{\ha}}\} = \{\cochain g_e^{\ha}\}_{e \in E_M \setminus \{e^*\}}$ \Comment{$N^{\ha}$ is equal to $\abs{E_M} - 1$}
    \EndIf
\end{algorithmic}
\label{alg:handles}
\end{algorithm}

\subsection{The second class of generators: the set $\mc G^\ho$}
The set $\mathcal G^{\ho}$ contains the generators associated with the connected components of $\partial \mathcal K$.
Let $\partial \mathcal K_k$, for $k \in \{1,\dots, N^{\ho}\}$, denote the connected components of $\partial \mathcal K$, and fix one of them, say $\partial \mathcal K_{N^{\ho}}$.
For each remaining component $\partial \mathcal K_k$ with $k \in \{1,\dots, N^{\ho}-1\}$, we construct a cohomology generator $\cochain g^{\ho}_k$ whose support consists precisely of those edges that have exactly one node in $\partial \mathcal K_k$.
\begin{algorithm}[H]
\caption{Construct the set $\mc G^{\ho}$}
\begin{algorithmic}[1]
    \For{each connected component $\partial \K_k$ of $\partial \K$ except $\partial \K_{N^{\ho}}$}
        \State define the $0$-cochain $\cochain c_k^{\ho}$ by
        \begin{align*}
        \inner{\cochain c_k^{\ho}}{n} \coloneqq
            \begin{cases}
            1 &\text{if } n \in \partial \K_k,\\
            0 &\text{if } n \notin \partial \K_k.
\end{cases}
        \end{align*}
        \Comment{$\cochain c_k^{\ho}$ is equal to $1$ on each boundary node $n$ of $\partial \K_k$ and $0$ elsewhere.}
        \State compute $\cochain g^{\ho}_k \coloneqq \delta^{0}\cochain{c}_k^{\ho}$ \Comment{Apply the coboundary map $\delta^0$ to $\cochain c_k^{\ho}$.}
    \EndFor
    \State \textbf{return} $\mc G^{\ho} \coloneqq \{\cochain g^{\ho}_1, \dots, \cochain g^{\ho}_{N^{\ho}-1}\}$
\end{algorithmic}
\label{alg:holes}
\end{algorithm}

\subsection{The third class of generators: the set $\mathcal{G}^\co$}
The set $\mathcal G^{\co}$ contains the generators associated with the connected components of the complementary set $\partial \K^\co \coloneqq \partial \K \setminus \partial 
\K^c$.
Let $\partial \mathcal K^{\co}_j$ for $j \in \{1,\dots, N^{\co}\}$ denote the connected components of $\partial \mathcal K^{\co}$, and fix one of them, say $\partial \mathcal K^{\co}_{N^{\co}}$.

For each remaining component $\partial \mathcal K^{\co}_j$ with $j \in \{1,\dots, N^{\co}-1\}$, we construct a generator $\cochain g^{\co}_j$ whose support consists precisely of the edges whose duals belong to a path in the dual graph $\dual G$ connecting $\partial \mathcal K^{\co}_j$ to $\partial \mathcal K^{\co}_{N^{\co}}$.
Moreover, depending on the cardinality of the set $E_M^{\mathrm{II}}$ introduced in \cref{alg:handles}, an additional generator must be included.
\begin{algorithm}[H]
\caption{Construct the set $\mc G^\co$}
\begin{algorithmic}[1]
    \State let $\dual T$ be the spanning tree of $\dual G'$ computed \cref{alg:handles} at line 3 \Comment{the spanning tree $\dual T$ has already been computed in \cref{alg:handles} (see the order in which the algorithms are called in \cref{alg:global.loops})}
    \State select an edge $e$ of $\partial \K^{\co}_{N^{\co}}$ and let $f_{e}$ be the unique face of $\K$ incident on it
    \For{each connected component $\partial \K^{\co}_j$ of $\partial \K^{\co}$ except $\partial \K^{\co}_{N^{\co}}$}
        \State select an edge $e_j$ of $\partial \K^{\co}_j$
   
      \State let $f_{e_j}$ be the unique face of $\K$ incident on $e_j$
        \State construct the unique path $\dual p_{e_j,e}$ in $\dual T$ from $D(f_{e_j})$ to $D(f_{e})$ \label{alg:contacts.line.path}
        \State apply \cref{alg:cocycle} with input $\dual p_{e_j,e}$ and $(e_j, e)$ to get $\cochain g_j^{\co}$
    \EndFor
    \State let $E_M^{\rm{II}}$ be the set computed in \cref{alg:handles} at line 14 \Comment{the set $E_M$ and its subset $E_M^{\rm{II}}$ have already been computed in \cref{alg:handles} (see comment at line 1)}
    
 \If{ $\abs{E_M^{\rm{II}}} = 0$}
        \State \textbf{return} $\mc G^{\co} \coloneqq \{\cochain g^{\co}_1, \dots, \cochain g^{\co}_{N^{\co}-1}\}$ 
    \Else \Comment{If $\abs{E_M^{\rm{II}}} > 0$}
        \State select an edge $e^* \in E_M^{\rm{II}}$
        \State let $f_{e^*}^{(1)}, f_{e^*}^{(2)}$ be the unique two faces of $\K$ incident on $e^*$
        \State construct the unique path $\dual p_{e,e^*}^{(1)}$ in $\dual T$ from $D(f_{e})$ to $D(f_{e^*}^{(1)})$
        \State construct the unique path $\dual p_{e,e^*}^{(2)}$ 
in $\dual T$ from $D(f_{e})$ to $D(f_{e^*}^{(2)})$
        \State apply \cref{alg:cocycle} with input $\dual p_{e,e^*}^{(1)}$ and $(e,e^*)$ to get $\cochain g_{e}^{(1)}$
        \State apply \cref{alg:cocycle} with input $\dual p_{e,e^*}^{(2)}$ and $(e,e^*)$ to get $\cochain g_{e}^{(2)}$
        \State define $\cochain g_{N^{\co}}^{\co}$ by
        \begin{align*}
        \inner{\cochain g_{N^{\co}}^{\co}}{e'} \coloneqq
            \begin{cases}            
 \inner{\cochain g_e^{(1)}}{e'}  &\text{if } e' = e^*,\\
            \inner{\cochain g_e^{(1)}}{e'} + \inner{\cochain g_e^{(2)}}{e'} &\text{if } e' \neq e^*.
\end{cases}
        \end{align*}
        \State \textbf{return} $\mc G^{\co} \coloneqq \{\cochain g^{\co}_1, \dots, \cochain g^{\co}_{N^{\co}}\}$
    \EndIf
\end{algorithmic}
\label{alg:contacts}
\end{algorithm}

\subsection{The main algorithm}
\label{sec:main.algo}
The main algorithm used to construct the complete set of generators $\mathcal{G}$ combines 
Algorithms~\ref{alg:handles}, \ref{alg:holes}, and \ref{alg:contacts}, and can be described as a three-step procedure.
\begin{algorithm}[H]
\caption{Construct the set $\mathcal{G}$}
\begin{algorithmic}[1]
    \State apply \cref{alg:handles} to obtain the set $\mathcal{G}^\ha$
    \If{$\partial \K \neq \emptyset$} \Comment{if $\K$ has a boundary}
        \State apply \cref{alg:holes} to obtain the set $\mathcal{G}^\ho$ 
        \If{$\partial \K^{\co} \neq \emptyset$} \Comment{if $\K^c$ is not $\partial \K$}
            \State apply \cref{alg:contacts} to obtain the set $\mathcal{G}^\co$ 
        \EndIf
    \EndIf
    \State \textbf{return} $\mathcal{G} \coloneqq \mathcal{G}^\ha 
\sqcup \mathcal{G}^\ho \sqcup \mathcal{G}^\co$ \Comment{return the union of all generators}
\end{algorithmic}
\label{alg:global.loops}
\end{algorithm}

In the next section, we provide the theoretical proof of correctness for \cref{alg:global.loops}.
\section{Proof of correctness of \cref{alg:global.loops}}
\label{sec:proof}
We divide the proof in nine steps.
\subsection{Tools from Discrete Morse Theory}
To start with, we recall a specific version of Discrete Morse Theory introduced by Kozlov in his book \cite{Kozlov2008}.

Let $\K$ be a CW complex, and for $m \in \{1, 2\}$, let $x$ be a $m$-cell and $x'$ be a $(m-1)$-cell of $\K$.
If $\iota(x,x')\neq 0$, we write $x' \prec x$.
\begin{definition}[Matching, acyclic matching]
\label{def:matching}
A \emph{matching} $M$ of $m$-cells on $\K$ is a set of pairs $(x,x')$ made of an $m$-cell and a $(m-1)$-cell of $\K$ such that:
\begin{enumerate}[label=(\roman*)]
\item If $(x,x') \in M$, then $x' \prec x$.
\item If $(x_1, x'), (x_2, x') \in M$, then $x_1 = x_2$.
\end{enumerate}
When $(x,x') \in M$, we write $x' = d(x)$.
A matching $M$ is \emph{acyclic} if there does not exist a cycle of matched elements $x_i$ for $i \in \{1, \dots l\}$ of the form
\begin{equation*}
x_1 \succ d(x_1) \prec x_2 \succ d(x_2) \prec \cdots \prec x_l \succ d(x_l) \prec x_1,
\end{equation*}
with $l \geq 2$.
\end{definition}

We say that an $m$-cell or $(m-1)$-cell of $\K$ is \emph{matched} in $M$ if it appears as the first or second member of a pair in $M$.
The set of \emph{critical cells} of $\K$ is set of all cells of $\K$ that are not matched in $M$.

The next lemma says that to every spanning tree on $\K$ we can associate an acyclic matching.
\begin{lemma}[Acyclic matching from spanning tree]
\label{lem:tree.acyclic}
Let $T$ be a spanning tree of $\K$.
Then, there exist an acyclic matching $M_T$ such that every vertex of $T$, except one (denoted $n$), is matched in $M_T$.
In particular, $n$ can be chosen arbitrarily. 
Similarly, let $\dual G=(\dual V, \dual E)$ be the dual graph of $\K$ and let $\dual T$ be a spanning tree on $\dual G$.
Then, there exist an acyclic matching $M_{\dual T}$ such that every dual vertex of $\dual T$, except one (denoted $\dual n$), is matched in $M_{\dual T}$.
In particular, $\dual n$ can be chosen arbitrarily. 
\end{lemma}
\begin{proof}
The proof of these two results follows directly from the reasoning in \cite[Section 4.2]{pitassi2022inverting} for the cases $k=0$ and $k=2$ (notice that the construction for the case $k=2$ is stated for $3$-dimensional domains but it can be easily adapted to $2$-dimensional ones).
\end{proof}

The next theorem can be regarded as the main theorem of Discrete Morse Theory for CW complexes, as it highlights the topological role of acyclic matchings.
\begin{theorem}\textup{\cite[Theorem 11.13 (b) and (c)]{Kozlov2008}}
\label{thm:main}
Let $M$ be an acyclic matching on $\K$, and denote by $d_m$ the number of critical $m$-cells of $\K$ with respect to $M$.
For every critical $m$-cell $x$ and $(m-1)$-cell $x'$ of $\K$, let $S(x,x')$ be the set of all sequences $s$ of matched cells of the form
\begin{equation*}
s \coloneqq (x \succ d(x_1) \prec x_1 \succ d(x_2) \prec x_2 \succ \dots \succ d(x_l) \prec x_l \succ x').
\end{equation*}
For each $s \in S(x,x')$, define $w(s)$ as the $m$-chain computed from the sequence $s$ as follows:
$$w(s) \coloneqq w_l,$$
where, for all $i \in \{0, \dots l\}$, $w_i$ is computed recursively as
\begin{equation}
\begin{cases}
w_0\coloneqq \xi_0\,x,\\
w_i \coloneqq w_{i-1} + \xi_{i}\, x_{i},
\end{cases}
\label{eq:recursion}
\end{equation}
with 
\begin{equation}
\begin{dcases}
\xi_0 \coloneqq 1 & \text{for } i = 0, \\[2pt]
\xi_i \coloneqq -\,\dfrac{\inner{\partial_{m+1} w_{i-1}}{d(x_i)}}{\inner{\partial_{m+1} x_i}{d(x_i)}} & \text{for } i \in \{1, \dots, l\}.
\end{dcases}
\label{eq:coefficient}
\end{equation}
Then, $\K$ is homotopy equivalent to a CW complex $\K_M$ with exactly $d_m$ cells of dimension $m$.
In particular, there is a natural identification of the cells of $\K_M$ with the critical cells of $\K$ such that, for any $x,x'$ as above, the incidence number $\iota^M(x,x')$ in $\K_M$ is given by
\begin{equation}
\iota^M(x,x') \coloneqq \sum_{s \in S(x,x')} \inner{\partial_m w(s)}{x'}.
\label{eq:morse.formula}
\end{equation}
\end{theorem}

\subsection{Properties of the $1$-cochain computed by \cref{alg:cocycle}}
Let $\dual{p}$ be a path in $\dual{G}$, and denote by $(f_1, \dots, f_l)$ the corresponding sequence of $2$-cells and by $(e_1, \dots, e_{l-1})$ the associated sequence of $1$-cells in $\mc{K}$.
Define $M_{\dual{p}}$ as the set of all pairs $(f_i, e_i)$ for $i \in \{1, \dots, l-1\}$.
It is straightforward to verify that $M_{\dual{p}}$ forms an acyclic matching on $\mc{K}$ such that $f_l$ is a critical $2$-cell in $\K$.

Accordingly, for each critical $1$-cell $e$ such that $e \subset f_1$, we define \emph{the sequence $s_{\dual{p}}$ associated with the path $\dual{p}$ in $M_{\dual{p}}$} as  
\begin{equation*}
s_{\dual{p}} \coloneqq (f_l \succ e_{l-1} \prec f_{l-1} \succ e_{l-2} \prec f_{l-2} \succ \dots \succ e_1 \prec f_1 \succ e).
\end{equation*}

\begin{lemma}[Properties of $1$-cochains of \cref{alg:cocycle}]
\label{lem:cocycle} 
Let $\dual p$ be a path in $\dual G$, and let $(e,e')$ be the pair of $1$-cells of $\K$ provided as input to \cref{alg:cocycle}, with $e \subset f_1$ and $e' \subset f_l$.
Let $\cochain g$ be the $1$-cochain of $\K$ computed by \cref{alg:cocycle}.
If $\cochain g \neq \cochain 0$, then $\inner{\cochain g}{e} = 1$, and we have
\begin{align}
\inner{\cochain g}{e'} &= -\frac{\iota(\xi_l f_l, e')}{\iota(\xi_1 f_1, e)},
\label{eq:transformation}
\end{align}
where the coefficients $\xi_1$ and $\xi_l$ are defined as in \cref{eq:coefficient} for the sequence $s_{\dual p}$, using the index substitution $i \leftarrow l-i$ for $i \in \{1, \dots, l\}$, so as to match the index of each $2$-cell $f_i$ along the path $\dual p$.
Furthermore, for all $i \in \{1, \dots, l-1\}$, the following relation holds:
\begin{equation}
\inner{\cochain g}{\partial_2(\xi_i f_i)} = 0.
\label{eq:cocycle.face}
\end{equation}
\end{lemma}
\begin{proof}
By the definition of the $\cochain g$ in \cref{alg:cocycle} at lines 1 and 4, we have that $\inner{\cochain g}{e}=1$, and for each $i \in \{1, \dots, l\}$, the following holds
\begin{align}
\inner{\cochain g}{e_{i}} = -\frac{\iota(f_{i},e_{i})}{\iota(f_{i}, e_{i-1})} \inner{\cochain g}{e_{i-1}},
\label{eq:transformation.b}
\end{align}
where we define $e_0\coloneqq e$, $e_l \coloneqq e'$ as the initial and final $1$-cells, respectively.
This formula expresses the transformation of the $1$-cochain as we move through the sequence of $1$-cells in $\dual p$.

Note that \cref{eq:transformation.b} is independent of the specific orientation of each $2$-cell $f_i$.
In fact, for any coefficient $\chi_i \in \{-1, +1\}$ corresponding to an oriented $2$-cell $\chi_i f_i$, we have $\iota(\chi_i f_i,e_{i}) = \chi_i \iota(f_i,e_{i})$.
This shows that the expression in \cref{eq:transformation.b} remains unchanged under the orientation change of the $2$-cells.

Next, we multiply each $2$-cell $f_i$ by the coefficient $\xi_i$, computed from the sequence $s_{\dual p}$ via \cref{eq:coefficient}.
Note that
\[
\iota(\xi_{i+1} f_{i+1},e_i) = -\iota(\xi_i f_i,e_i)
\]
for all $i\in\{1,\dots,l-1\}$; this follows from the choice of coefficients $\xi_i$ in \cref{eq:coefficient} and from the recursion formula \cref{eq:recursion}.
Therefore, applying \cref{eq:transformation.b} recursively yields the desired result \cref{eq:transformation}.

To prove \cref{eq:cocycle.face}, we rewrite \cref{eq:transformation.b} as the following equation:
\begin{equation*}
\inner{\cochain g}{e_i} \iota(\xi_i f_{i},e_i) + \inner{\cochain g}{e_{i-1}} \iota(\xi_i f_{i}, e_{i-1}) = 0,
\end{equation*}
for each $i \in \{1, \dots, l-1\}$.
This is simply a restatement of the transformation rule for the $1$-cochain.

Finally, recall that $\cochain g = \cochain 0$ as initialized in \cref{alg:cocycle} at line 1, and note that the $2$-cells $f_i$ are distinct since $\dual p$ is a path.
By linearity and by the definition of the boundary map $\partial_{2}$ we thus have
\[
\inner{\cochain g}{\iota(\xi_i f_{i},e_i)\,e_i + \iota(\xi_i f_{i}, e_{i-1})\,e_{i-1}}
= \inner{\cochain g}{\partial_2(\xi_i f_i)} = 0,
\]
which is the desired equation in \cref{eq:cocycle.face}.
\end{proof}

\subsection{Definition of a new CW complex $\bar{\K}$ with empty boundary and its reduction to a homotopy-equivalent CW complex $\bar \K_M$}
\label{sec:reduction}
We first define the CW complex $\bar \K$, constructed from $\K$ as follows.

For each connected component of the boundary $\partial \K_k$ with $k \in \{1, \dots, N^{\ho}\}$, we do the following operation: we consider a gluing map (in the CW sense as in \cref{def:CW_complex}) attaching a closed disk $D_k$ (topologically a $2$-ball) along the (polygonal) boundary $\partial \K_k$, and we contract $D_k$ to a single $0$-cell, say $n_k$.
Accordingly, the CW complex structure of $\bar \K$ is obtained from that of $\K$ by removing every $1$-cell of $\partial \K$ and by identifying, for each $k \in \{1, \dots, N^{\ho}\}$, every $0$-cell in $\partial \K_k$ with the $0$-cell $n_k$ in $\bar \K$.

Let $T$ be the spanning tree of $\K$ defined in \cref{alg:handles}, lines 1--2, and let $\dual T$ be the tree of the dual graph $\dual G$ of $\K$, defined in lines 3--7 of the same algorithm.
These trees also induce corresponding spanning trees of $\bar \K$ and of its dual graph $\dual G$ (the latter because $\bar \K$, by construction, has empty boundary), which we denote by the same symbols for simplicity.

Let $ M_T $ and $ M_{\dual T} $ be the acyclic matchings constructed from the spanning trees $ T $ and $ \dual T $, respectively, by applying \cref{lem:tree.acyclic} to $ \bar \K $.

Applying \cref{thm:main} to $ \bar \K $ and the acyclic matching $ M = M_{\dual T} \cup M_T $, we obtain a CW complex $ \bar \K_M $ that is homotopy equivalent to $ \bar \K $, with its cells naturally indexed by the set of critical cells of $ \bar \K $ with respect to $M$.

Since $ T $ is a spanning tree of $ \bar \K $, each edge of $ T $ is matched in $ M_T $ to a vertex (a $0$-cell) of $ \bar \K $.
As a result, $ M_T $ consists of $ \abs{\bar \K^{(0)}} - 1 $ pairs.
Similarly, since $ \dual T $ is a spanning tree of $ \dual G $, each edge of $ \dual T $ is matched in $ M_{\dual T} $ to a vertex (a $2$-cell, by duality) of $ \dual G $.
Therefore, $ M_{\dual T} $ consists of $\abs{\bar \K^{(2)}} - 1 $ pairs.

Consequently, $ \bar \K_M $ consists of:
\begin{itemize}
    \item One $0$-cell, which is the unique non-matched vertex of $ \bar \K $ in $ M_T $.
\item $ \abs{\bar \K^{(1)}} - (\abs{\bar \K^{(0)}} - 1) - (\abs{\bar \K^{(2)}} - 1) = 2 - \chi $ $1$-cells, where $ \chi $ is the Euler characteristic of $ \bar \K $.
\item One $2$-cell, which is the unique non-matched vertex (again, by duality) of $ \dual G $ in $ M_{\dual T} $.
\end{itemize}

It is important to note that the set of all $1$-cells of $ \bar \K_M $ is in one-to-one correspondence with the set $ E_M $, as defined in \cref{alg:handles} at line 8. For the remainder of this paper, we will identify these two sets.

Now, let $ f $ denote the unique $2$-cell of $ \bar \K_M $, and consider the expression for the boundary map $ \partial^M_2 $ on the CW complex $ \bar \K_M $ with incidence numbers $ \iota^M(f, e) $, as given by \cref{eq:morse.formula}.
Specifically, we have
\begin{equation}
\partial^M_2 f = \sum_{e \in E_M} \iota^M(f, e) \, e,
\label{eq:expansion}
\end{equation}
where $ \iota^M(f, e) $ is the incidence number associated with the boundary of $ f $ along the $1$-cell $e$.

Since $ \bar \K_M $ is homotopy equivalent to $ \bar \K $ (by \cref{thm:main}) and the boundary of $ \bar \K $ is empty, we conclude that each coefficient $ \iota^M(f, e) $ in the above sum can only take one of two possible values:
\begin{enumerate}
    \item $ \iota^M(f, e) = 0 $.
\item $ \iota^M(f, e) = 2\eta_e $, where $ \eta_e \in \{-1, +1\} $.
\end{enumerate}
Therefore, we can partition $ E_M $ into two subsets: $ E_M = E_M^{\rm I} \sqcup E_M^{\rm II} $, where $ E_M^{\rm I} $ and $ E_M^{\rm II} $ consist of the edges $ e \in E_M $ that satisfy conditions 1. and 2., respectively.

\subsection{Differentiation between homology and torsion generators of $H_1(\bar \K;\Z)$ passing through $\bar \K_M$}
\label{sec:differentiation}
The characterization of the homology of $\bar \K$ is an immediate consequence of the Classification Theorem of Compact Surfaces \cite{gallier2013guide}.
In the next Lemma we provide an equivalent characterization which will be useful for our specific purposes.
\begin{lemma}[Homology characterization of $\bar \K$]
\label{lem:homology}
For each $e \in E_M$, let $c_e$ be the unique cycle in the graph $T \cup \{e\}$, naturally oriented according to the orientation induced by $e$.
Note that $c_e$ can considered as a $1$-cycle in $\bar{\K}$.
Let $E_M = E_M^{\rm I} \sqcup E_M^{\rm II}$, and denote by $g := \abs{E_M} = 2 - \chi$ the genus of $\bar \K_M$.
We distinguish between the following two cases:
\begin{enumerate}[label=(\roman*)]
\item If $\abs{E_M^{\rm II}} = 0$, then $H_1(\bar \K; \Z) \cong \Z^g$.
\item If $\abs{E_M^{\rm II}} > 0$, then $H_1(\bar \K;\Z)\cong\Z^{g-1} \oplus \Z/2\Z$.
In this case, the torsion generator $\torsion c$ is given by 
\begin{equation}
\torsion c := \sum_{e \in E_M^{\rm{II}}} \eta_e\, c_e,
\label{eq:torsion}
\end{equation}
where each $\eta_e \in \{-1,+1\}$ is the sign of the coefficient $\iota^M(f,e)$ in the expression of $\partial^M_2 f$ as a linear combination of $1$-cells $e \in E_M^{\rm{II}}$ in \cref{eq:expansion}.
\end{enumerate}
\end{lemma}
\begin{proof}
Since $\bar{\K}_M$ has only one $0$-cell, every $1$-cell $e \in E_M$ forms a $1$-cycle.

In the first case, when $\abs{E_M^{\mathrm{II}}} = 0$, we have $E_M = E_M^{\mathrm{I}}$.
Therefore, each $1$-cell $e \in E_M$ is a homology generator of $H_1(\bar{\K}_M; \mathbb{Z})$, since $\partial_2^M f \neq e$.
Using the fact that $\bar{\K}$ is homotopy equivalent to $\bar{\K}_M$, we conclude that each $1$-cycle $c_e$ is also a homology generator of $H_1(\bar{\K}; \mathbb{Z})$.

In the second case, where $\abs{E_M^{\mathrm{II}}} > 0$, fix a $1$-cell $e^* \in E_M^{\mathrm{II}}$.
The boundary operator $\partial_2^M f$ satisfies
\[
\partial_2^M f = \sum_{e \in E_M^{\mathrm{II}}} 2 \eta_e e \eqqcolon 2c^*,
\]
where $c^*$ is the (unique) torsion generator of $\bar{\K}_M$.

For each $e \in E_M \setminus \{e^*\}$, the $1$-cell $e$ is a homology generator of $H_1(\bar{\K}_M; \Z)$, since $\partial_2^M f \neq e$.
This means there are $g - 1$ independent homology generators.

Next, let $c_e$ be the unique $1$-cycle of $\bar{\K}$ associated with each $e \in E_M$, as in the case where $\abs{E_M^{\mathrm{II}}} = 0$, and define the $1$-cycle $\torsion c$ as in \cref{eq:torsion}.
Using again the fact that $\bar{\K}$ is homotopy equivalent to $\bar{\K}_M$, each $c_e$ for $e \in E_M \setminus \{e^*\}$ is a homology generator of $H_1(\bar{\K}; \mathbb{Z})$.
Furthermore, $\torsion c$ is the torsion generator of $H_1(\bar{\K}; \mathbb{Z})$, since $\torsion c$ is homotopy equivalent to the torsion $1$-cycle $c^*$ in $\bar{\K}_M$.
\end{proof}

\subsection{Characterization of generators of $H^1(\bar \K)$}
\label{sec:characterization}
Let us remind the Universal Coefficient Theorem for cohomology.
\begin{theorem}[{\cite[Theorem 3.2]{hatcher}}]
\label{thm:universal}
If a chain complex $C$ of free Abelian groups has homology groups $H_n(C;G)$, then the cohomology groups $H^n(C;G)$ of the cochain complex $Hom(C;G)$ are determined by the exact sequence
\begin{equation*}
0 \xrightarrow{} \mathrm{Ext}(H_{n-1}(C;G);G) \xrightarrow{} H^n(C;G) \xrightarrow{h} \mathrm{Hom}(H_n(C;G);G) \xrightarrow{} 0.
\end{equation*}
\end{theorem}
For our purposes, we do not need to delve into the definition of the $\mathrm{Ext}$ functor.
It suffices to note the following property: $\mathrm{Ext}(Q;G) = 0$ when $Q$ is a free group.
For further details and a proof of this property, consult \cite{hatcher}.

We now apply \cref{thm:universal} to $\bar{\K}$ with $n = 1$ and $G = \Z$.
In \cref{thm:universal}, the map $h: H^1(\bar{\K};\Z) \to \mathrm{Hom}(H_1(\bar{\K};\Z); \Z)$ assigns to each generator $\cochain g \in H^1(\bar{\K};\Z)$ the function $\inner{\cochain g}{\cdot}$, which maps an element $c \in H_1(\bar{\K};\Z)$ to the value $\inner{\cochain g}{c}$.

Note that if $\mathrm{Ext}(H_0(\bar{\K}; \Z); \Z) \cong 0$, then $h$ is an isomorphism.
Since we assume $\mc{\K}$ is connected, $\bar{\K}$ is also connected, implying that $H_0(\bar{\K};\Z) \cong \Z$.
Consequently, $\mathrm{Ext}(H_0(\bar{\K};\Z); \R) \cong 0$, and from the exactness of the sequence in \cref{thm:universal}, we conclude that $h : H^1(\bar{\K};\Z) \to \mathrm{Hom}(H_1(\bar{\K};\Z); \Z)$ is an isomorphism.

As a result, we obtain the following characterization of generators of $H^1(\bar{\K};\R)$ with real coefficients in terms of a set of generators of $H_1(\bar{\K}; \Z)$.
\begin{corollary}[Characterization of cohomology generators]
\label{cor:characterization}
Let $\{c_i\}_{i=1}^{\beta_1}$ be a set of homology generators of $H_1(\bar \K; \Z)$, and let $\{\cochain g_i\}_{i=1}^{\beta_1}$ be a set of $1$-cocycles of $\bar \K$ such that
\begin{align}
\inner{\cochain g_i}{c_j} &= \delta_{i,j},
\label{eq:universal}
\end{align}
for all $i,j \in \{1, \dots, \beta_1\}$, where $\delta_{i,j}$ is the Kronecker delta.
Then, the set $\{\cochain g_i\}_{i=1}^{\beta_1}$ forms a basis of $H^1(\bar \K; \Z)$.
In particular, the same set yields a basis for $H^1(\bar \K; \R)$.
\end{corollary}

\subsection{A first decomposition and proof of correctness of \cref{alg:handles}} 

Recall the definition of the trees $T$ and $\dual T$ defined in \cref{alg:handles} (lines 1--7).
Each connected component $\partial\K_k$ of the boundary $\partial\K$ is a $1$-cycle with $\abs{\partial\K_k^{(1)}}$ edges, and any spanning tree of such a cycle has $\abs{\partial\K_k^{(1)}}-1$ edges, hence $T$ contains all but one edge of each $\partial\K_k$, justifying the choice at line 5. Therefore no boundary edge belongs to the set $E_M$, validating the selection at line 10.

Since $\dual T$ is built from dual edges associated with edges not in $E_M$, and each path $\dual p_e$ in $\dual T$, defined at lines 10--11, connects dual nodes associated with faces, it follows that each $\dual p_e$ contains no boundary edge of $\K$.
By the formula in \cref{alg:cocycle} at line 4, the support of any $1$-cochain $\cochain g$ produced from these paths contains no boundary edges of $\K$.
Consequently, if $\cochain g$ is a $1$-cocycle on $\K$, then it is also a $1$-cocycle on $\bar\K$.
\begin{figure}   
\includegraphics[width=1.5\textwidth,center]{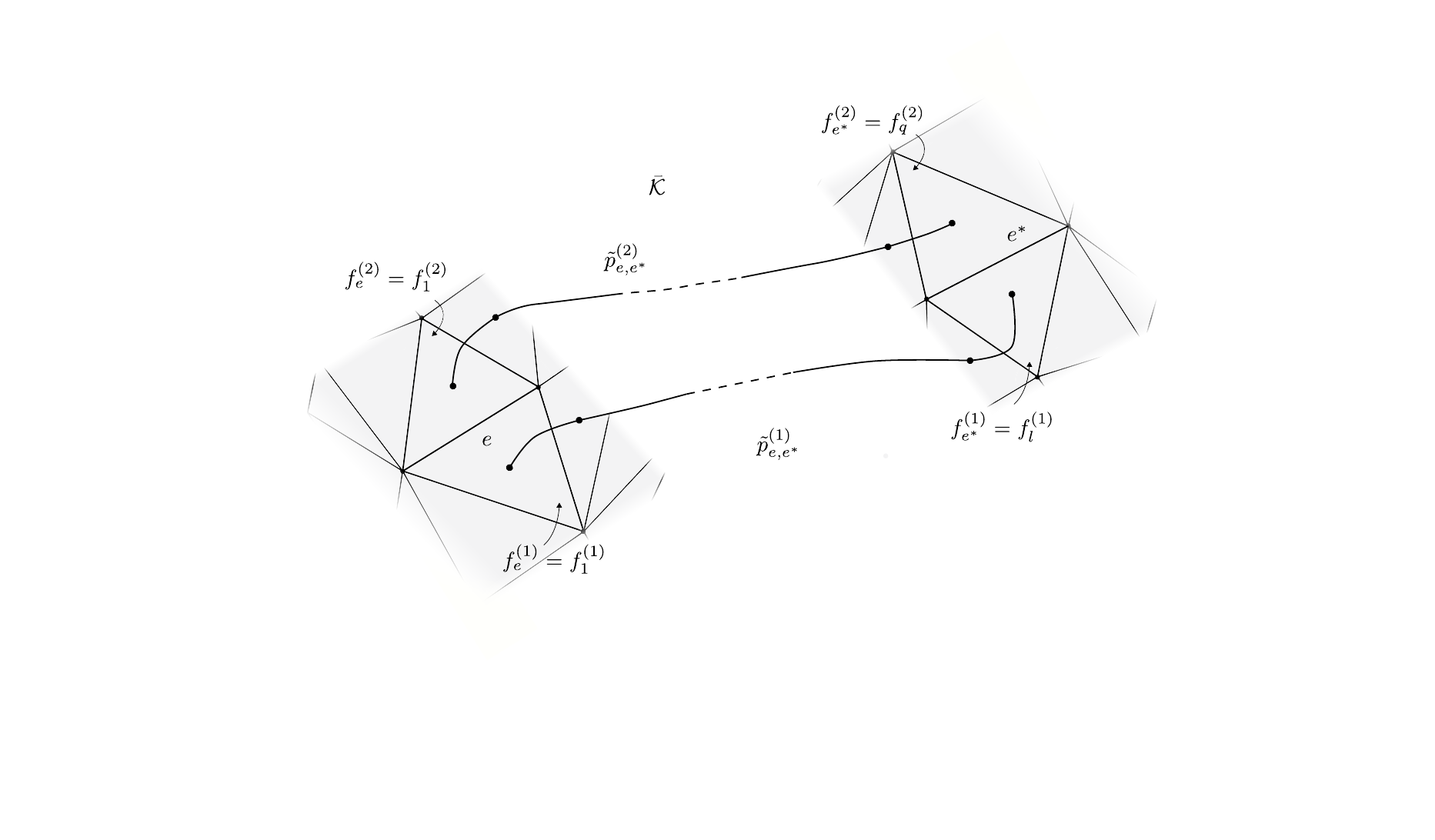}
\caption{Illustration of the notation used in the proof of Point (ii) in \cref{lem:equivalence} and in \cref{alg:handles}.}
\label{fig:dual}
\end{figure}

\begin{lemma}[Properties of $1$-cochains of \cref{alg:handles}]
\label{lem:equivalence}
Recall the definition of the set $E_M$ in \cref{alg:handles}, line 8, and consider the partition $E_M = E_M^{\rm I} \sqcup E_M^{\rm II}$ introduced at the end of \cref{sec:reduction}.
Let $e \in E_M$.
The following two assertions hold:
\begin{enumerate}[label=(\roman*)]
\item Let $\dual p_e$ be the path defined in \cref{alg:handles} at lines 10 and 11, and let $\cochain g_e^{\ha}$ be the corresponding $1$-cochain of $\bar \K$ computed at line 12 using \cref{alg:cocycle}.
Then, $e \in E_M^{\rm I}$ if and only if $\cochain g_e^{\ha}$ is a $1$-cocycle of $\bar \K$ whose restriction to $\bar \K_M$ satisfies the condition
\begin{equation}
\inner{\cochain g_e^{\ha}}{e'}=\delta_{e,e'}
\label{eq:condition.a}
\end{equation}
for every $e' \in E_M$.
\item Assume $|E_M^{\rm{II}}|>0$, and let $e^*$ the $1$-cell selected in \cref{alg:handles} at line 18.
For each $e \in E_M^{\rm{II}} \setminus\{e^*\}$, let $\dual p_{e,e^*}^{(1)}, \dual p_{e,e^*}^{(2)}$ be the paths defined in \cref{alg:handles}, at lines 20--23.
Let $\cochain g_e^{(1)}$ and $\cochain g_e^{(2)}$ denote the corresponding $1$-cochains of $\bar \K$ computed at lines 24 and 25 using \cref{alg:cocycle}, and let $\cochain g_e^{\ha}$ be the $1$-cochain defined at line 26.
Then, $\cochain g_e^{\ha}$ is a $1$-cocycle of $\bar \K$ whose restriction to $\bar \K_M$ satisfies the condition
\begin{equation}
\inner{\cochain g_e^{\ha}}{e'}=\delta_{e,e'}
\label{eq:condition.b}
\end{equation}
for every $e' \in E_M^{\rm{II}} \setminus\{e^*\}$.
\end{enumerate}
\end{lemma}
\begin{proof} \smallskip
(i) Let $e \in E_M$.
Denote by $(f_1, \dots, f_l)$ the sequence of $2$-cells of $\dual p_e$, and by 
$(e_1, \dots, e_{l-1})$ the associated sequence of $1$-cells.
Notice that we can identify the unique $2$-cell $f$ of $\bar \K_M$ by setting $f := f_l$.
This follows from the fact that every $2$-cell of $\K$, except one, is matched in the acyclic matching $M$, and that the unique non-matched $2$-cell in $M$ can be chosen arbitrarily (as discussed in the property of $M_{\dual T}$ in \cref{lem:tree.acyclic}).
This also matches with the natural identification of cells of $\bar \K_M$ with the critical cells of $\bar \K$ as stated in \cref{thm:main}.

Let $s_{\dual p_e}$ be the sequence associated with the path $\dual p_e$ in $M$.
Applying Point (ii) in \cref{thm:main} to the pair $f$, $e$, we observe that $s_{\dual p_e}$ is the unique sequence in the set $S(f,e)$.
This is because $f_1,f_l$ are the only $2$-cells of $\bar \K$ incident on $e$, with $f:=f_l$.

Next, applying formula \cref{eq:morse.formula}, we conclude that $e \in E_M^{\rm I}$ if and only if
\[\iota^{M}(f,e)= \iota(\xi_1 f_1,e) + \iota(\xi_{l} f_l,e)=0.\]
Thus, we obtain the following condition
\begin{equation}
\iota(\xi_1 f_1,e)=-\iota(\xi_{l}f_l,e).
\label{eq:compatibility.condition.b}
\end{equation}

By \cref{lem:cocycle}, we know that the $1$-cochain $\cochain g_e^{\ha}$ satisfies $\inner{\cochain g_e^{\ha}}{e}=1$ and $\inner{\cochain g_e^{\ha}}{\partial_2(\xi_i f_i)} = 0$ for all $i \in \{1, \dots, l-1\}$.
Since $e'=e$ in \cref{alg:handles} at line 12, it follows that $\cochain g_e^{\ha}$ is a $1$-cocycle if and only if $\inner{\cochain g_e^{\ha}}{\partial_2(\xi_l f_l)} = 0$, which leads to the compatibility condition 
\begin{equation*}
1=\inner{\cochain g^{\ha}}{ e}=\inner{\cochain g^{\ha}}{e'}.
\end{equation*}
Using \cref{eq:transformation} from \cref{lem:cocycle}, we conclude that this condition is exactly \cref{eq:compatibility.condition.b}.
\smallskip\\
(ii) Let $(f_1^{(1)}, \dots, f_l^{(1)})$ and $(e_1^{(1)}, \dots, e_{l-1}^{(1)})$ be the sequences of $2$-cells and $1$-cells of $\dual p_{e,e^*}^{(1)}$, and let $(f_{1}^{(2)}, \dots, f_q^{(2)})$ and $(e_{1}^{(2)}, \dots, e_{q-1}^{(2)})$ be the corresponding sequences for $\dual p_{e,e^*}^{(2)}$.
See \cref{fig:dual} for reference.

Let $\dual p$ be the unique path in $\dual T$ from $D(f_1^{(2)})$ to $D(f_l^{(1)})$ and let $s_{\dual p}$ be the corresponding sequence in $M$.
As in Point (i), we can identify $f \coloneqq f_l^{(1)}$.

Apply Point (ii) of \cref{thm:main} to the pair $f$, $e$, and observe that $s_{\dual p}$ and $s_{\dual p_{e,e^*}^{(1)}}$ are the only two sequences in the set $S(f,e)$, since $f_1^{(1)},f_1^{(2)}$ are the unique $2$-cells of $\bar \K$ incident on $e$.
Then, applying formula \cref{eq:morse.formula} to the sequences $s_{\dual p_{e,e^*}^{(1)}}$ and $s_{\dual p}$, we have \[\iota^{M}(f,e)= \iota(\xi_1^{(1)} f_1^{(1)},e) + \iota(\xi_{1}^{(2)} f_{1}^{(2)},e).\]
Since $e \in E_M^{\rm{II}}$, by the definition of $E_M^{\rm{II}}$ in \cref{sec:differentiation}, we have $\iota^{M}(f,e) = 2 \eta_e$ with $\eta_e \in \{-1,1\}$.
Moreover, since both $2$-cells $f_{1}^{(1)}$, $f_{1}^{(2)}$ are also $2$-cells of $\K$, we deduce that $\iota(\xi_1^{(1)} f_1^{(1)},e), \iota(\xi_{1}^{(2)} f_{1}^{(2)},e) \in \{-1,1\}$.
Therefore, we conclude that \[\iota(\xi_1^{(1)} f_1^{(1)},e)=\iota(\xi_{1}^{(2)} f_{1}^{(2)},e)= \eta_e.\]

Similarly, let $\dual p$ be the unique path in $\dual T$ from $D(f_q^{(2)})$ to $D(f_l^{(1)})$.
Apply again Point (ii) in \cref{thm:main} to the pair $f$, $e^*$, and observe that $s_{\dual p}$ is the unique sequence of the set $S(f,e)$, since $f_l^{(1)}$ and $f_q^{(2)}$ are the only $2$-cells of $\bar \K$ incident on $e$.

Then, applying formula \cref{eq:morse.formula} to the sequence $s_{\dual p}$, we have \[\iota^{M}(f,e^*)=\iota(\xi_l^{(1)} f_{l}^{(1)},e^*) + \iota(\xi_{q}^{(2)} f_{q}^{(2)},e^*).\]
Since $e^* \in E_M^{\mathrm{II}}$, we conclude, using similar reasoning as above, that \[\iota(\xi_l^{(1)} f_{l}^{(1)},e^*) = \iota(\xi_{q}^{(2)} f_{q}^{(2)},e^*) = \eta_{e^*}\] with $\eta_{e^*} \in \{-1,1\}$.

By \eqref{eq:cocycle.face} in \cref{lem:cocycle}, we have that $\inner{\cochain g_e^{(1)}}{\partial_2(\xi_i^{(1)} f_i^{(1)})} = 0$ for all $i \in \{1, \dots, l-1\}$, and $\inner{\cochain g_e^{(2)}}{\partial_2(\xi_i^{(2)} f_i^{(2)})} = 0$ for all $i \in \{1, \dots, q-1\}$.
Recalling the definition of $\cochain g_e^{\ha}$ in \cref{alg:handles} at line 26, it follows that $\cochain g_e^{\ha}$ is a $1$-cocycle if and only if $\inner{\cochain g_e^{(1)}}{\partial_2(\xi_l^{(1)} f_l^{(1)})} = 0$ and $\inner{\cochain g_e^{(2)}}{\partial_2(\xi_q^{(2)} f_q^{(2)})} = 0$.
These conditions lead to the compatibility condition 
\begin{equation*}
\inner{\cochain g_e^{(1)}}{e^*} = \inner{\cochain g_e^{(2)}}{e^*},
\end{equation*}
along with $\inner{\cochain g_e^{(1)}}{e} = \inner{\cochain g_e^{(2)}}{e}$, which holds by the initialization in \cref{alg:cocycle} at line 2.

Applying formula \cref{eq:transformation} in \cref{lem:cocycle} to the sequences $s_{\dual p_{e,e^*}^{(1)}}$, $s_{\dual p_{e,e^*}^{(2)}}$, we deduce that
\begin{gather*}
\inner{\cochain g_e^{(1)}}{e^*}
=-\frac{\iota(\xi_{l}^{(1)} f_{l}^{(1)},e^*)}{\iota(\xi_1^{(1)} f_1^{(1)},e) }
=-\frac{\eta_{e^*}}{\eta_e}
=-\frac{\iota(\xi_{q}^{(2)} f_q^{(2)}, e^*)}
{\iota(\xi_{1}^{(2)} f_{1}^{(2)}, e)}
=\inner{\cochain g_e^{(2)}}{e^*}.
\end{gather*}

Thus, $\cochain g_e^{\ha}$ is a $1$-cocycle, and its restriction to $\bar \K_M$ satisfies the desired condition, as stated in \cref{eq:condition.b}.
\end{proof}

\begin{lemma}[First decomposition of global loops]
\label{lem:decomposition.a}
Let $\mc G^{\ha} = \{\cochain g_i^{\ha}\}_{i=1}^{\abs{\mc G^{\ha}}}$ be the set of $1$-cochains computed by \cref{alg:handles}, and let $\eqclassspan{\mc G^{\ha}} \coloneqq \{ \sum_{i=1}^{\abs{\mc G^{\ha}}} \gamma_i [\cochain g_i^{\ha}] \mid \gamma_i \in \R\}$ be the linear span of the equivalence classes associated with each generator $\cochain g^{\ha}_i \in \mc G^{\ha}$.
Then, the following isomorphism holds
\begin{equation*}
H^1(\bar \K) \cong \eqclassspan{\mc G^{\ha}}.
\label{eq:decomposition.a}
\end{equation*}
\end{lemma}
\begin{proof}
For each $e \in E_M$, let $\cochain g_e^{\ha}$ be the $1$-cocycle of $\bar \K$ computed in \cref{alg:handles} at line 12.
Applying Point (i) in \cref{lem:equivalence}, it follows that, if $\cochain g_e^{\ha} \neq \cochain 0$, then $\cochain g_e^{\ha}$ is such that its restriction to $\bar \K_M$ satisfies \cref{eq:condition.a}, and $e \in E_M^{\rm I}$.
On the other hand, if $\cochain g_e^{\ha} = \cochain 0$, then $e \in E_M^{\rm{II}}$.
This defines the partition $E_M = E_M^{\mathrm {I}} \sqcup E_M^{\mathrm {II}}$, and we now analyze the two possible cases for this partition.

In the case where $\abs{E_M^{\mathrm{II}}} = 0$, observe that, by construction, the support of $\cochain g_e^{\ha}$ contains no edge of $T$.
Consequently, for each $e' \in E_M$, letting $c_{e'}$ denote the corresponding $1$-cycle in $\bar{\K}$ (as defined in the statement of \cref{lem:homology}), we have  
\[
\inner{\cochain g_e^{\ha}}{c_{e'}} = \inner{\cochain g_e^{\ha}}{e'} = \delta_{e,e'},
\]
by virtue of \cref{eq:condition.a}.
Then, applying \cref{cor:characterization} together with \cref{lem:homology}, we conclude that $\mc G^{\ha} = \{\cochain g_e^{\ha}\}_{e \in E_M}$ constitutes a set of generators of $H^1(\bar{\K})$.

In the case where $\abs{E_M^{\mathrm {II}}} > 0$, let $e^* \in E_M^{\mathrm {II}}$ denote the $1$-cell fixed in \cref{alg:handles} at line 18. 
For each $e \in E_M^{\mathrm {II}} \setminus \{e^*\}$, let $\cochain g_e^{\ha}$ be the corresponding $1$-cocycle of $\K$ computed in \cref{alg:cocycle} at line 26.
By Point (ii) of \cref{lem:equivalence}, each $\cochain g_e^{\ha}$ for $e \in E_M^{\rm II} \setminus \{e^*\}$ is $1$-cocycle of $\bar \K$ whose restriction to $\bar \K_M$ satisfies \cref{eq:condition.b}.
Since the support of $\cochain g_e^{\ha}$ does not contain any edge of $T$, it follows that, for all $e' \in E_M \setminus\{e^*\}$, we have
\begin{equation*}
\inner{\cochain g_e^{\ha}}{c_{e'}} = \inner{\cochain g_e^{\ha}}{e'}=\delta_{e,e'},
\end{equation*}
thanks to \cref{eq:condition.b}, as well as $\inner{\cochain g_e^{\ha}}{\torsion c} = \inner{\cochain g_e^{\ha}}{ c^*}$.

Using \cref{cor:characterization} and \cref{lem:homology}, we conclude that $\mc G^{\ha} = \{\cochain g_e^{\ha}\}_{e \in E_M \setminus \{e^*\}}$ is a set of generators of $H^1(\bar \K)$.

Thus, in both cases, the set $\mc G^{\ha}$ provides a valid set of generators of $H^1(\bar \K)$, determined by analyzing the contributions of $1$-cocycles based on the partition $E_M = E_M^{\mathrm {I}} \sqcup E_M^{\mathrm {II}}$.
\end{proof}

\subsection{A second decomposition and proof of correctness of \cref{alg:holes}}
\label{sec:second.decomposition}

We now focus on $H^1(\K, \partial\K)$ and present a preparatory result for the second decomposition of global loops, which is associated with the connected components of the boundary $\partial \K$.
\begin{lemma}[Dimension calculation via long exact sequences]
\label{lem:dimension.a}
Let $\K$ a $2$-dimensional CW complex with a non-empty boundary $\partial \K$.
Then, the relationship between the dimensions of $H^1(\K, \partial \K)$ and $H^1(\bar \K)$ is given by the following expression
\begin{equation}
\rank{H^1(\K, \partial \K)} - \rank{H^1(\bar \K)} = N^{\ho} -1.
\label{eq:dimension.a}
\end{equation}
\end{lemma}
\begin{proof}
The proof consists in expressing the dimensions of $H^1(\K, \partial \K)$ and $H^1(\bar \K)$ in terms of that $H^1(\K)$ and in computing the dimension of some elementary cohomology spaces.

We start by considering the Mayer-Vietoris long exact sequence \cite[Theorem 3, Chapter 11]{spivak1979comprehensive} for the triad $(\bar \K$, $\K$, $D)$, where $D$ is defined as $D \coloneqq \bigcup_{k=1}^{N^{\ho}} D_k$.
The exact sequence is as follows:
\begin{align}
\begin{split}
0 &\longrightarrow \underset{1}{\underbrace{H^0(\bar \K)}} \xrightarrow{} \underset{1}{\underbrace{H^0(\K)}}\oplus \underset{N^{\ho}}{\underbrace{H^0(D)}} \xrightarrow{} \underset{N^{\ho}}{\underbrace{H^0(\K \cap D)}} \rightarrow\\
&\xrightarrow{} H^1(\bar \K) \xrightarrow{} H^1(\K) \oplus \underset{0}{\underbrace{H^1(D)}} \xrightarrow{} \underset{N^{\ho}}{\underbrace{H^1(K \cap D)}} \rightarrow\\
&\xrightarrow{} H^2(\bar \K) \xrightarrow{} \underset{0}{\underbrace{H^2(\K)}} \oplus \underset{0}{\underbrace{H^2(D)}} \xrightarrow{} \underset{0}{\underbrace{H^2(\K \cap D)}} \longrightarrow 0,
\label{eq:sequence.a.1}
\end{split}
\end{align}
where the dimensions of the corresponding cohomology groups are indicated under the braces.

Next, we apply the long exact sequence in relative cohomology \cite[Theorem 13]{spivak1979comprehensive} for the pair $(\K, \partial \K)$:
\begin{align}
\begin{split}
0 &\longrightarrow \underset{0}{\underbrace{H^0(\K, \partial \K)}} \xrightarrow{} \underset{1}{\underbrace{H^0(\K)}} \xrightarrow{} \underset{N^{\ho}}{\underbrace{H^0(\partial \K)}} \rightarrow\\
&\xrightarrow{} H^1( \K, \partial \K) \xrightarrow{} H^1(\K) \xrightarrow{} \underset{N^{\ho}}{\underbrace{H^1(\partial \K)}} \rightarrow\\
&\xrightarrow{} H^2(\K, \partial \K) \xrightarrow{} \underset{0}{\underbrace{H^2(\K)}} \xrightarrow{} \underset{0}{\underbrace{H^2(\partial \K)}} \longrightarrow 0,
\label{eq:sequence.a.2}
\end{split}
\end{align}
where the dimensions of the cohomology groups are indicated under the braces.
In particular, $N^{\ho}$ denotes the number of connected components of $\partial \K$.

Finally, applying the Euler characteristic formula \cite[Theorem 2.44]{hatcher} to the exact sequences \cref{eq:sequence.a.1} and \cref{eq:sequence.a.2}, and noting that in \cref{eq:sequence.a.2}, the dimension of $H^2(\K, \partial \K)$ equals that of $H^2(\bar \K)$ in \cref{eq:sequence.a.1}, regardless of whether $\K$ is orientable, we obtain the desired formula \cref{eq:dimension.a}.
\end{proof}

We are now ready to prove the second decomposition of global loops, which demonstrates the correctness of \cref{alg:holes}.
\begin{lemma}[Second decomposition of global loops]
\label{lem:decomposition.b}
Let $\mc G^{\ho} = \{\cochain g^{\ho}_k\}_{k=1}^{\abs{\mc G^{\ho}}}$ be the set of $1$-cochains computed by \cref{alg:holes}, and define $\eqclassspan{\mc G^{\ho}} \coloneqq \{\sum_{k=1}^{\abs{\mc G^{\ho}}} \beta_k [\cochain g^{\ho}_k] \mid \beta_k \in \R\}$ as the linear span of the equivalence classes associated with each generator $\cochain g^{\ho}_k \in \mc G^{\ho}$.
Then, the following group isomorphism holds
\begin{align}
H^1(\K, \partial \K) \cong  H^1(\bar \K) \oplus \eqclassspan{\mc G^{\ho}}.
\label{eq:decomposition.b}
\end{align}
\end{lemma}
\begin{proof}
We begin by recalling the definition of $\cochain g_k^{\ho}$ in \cref{alg:holes}, specifically at lines 2--3.
Each $\cochain g_k^{\ho}\in \mc G^{\ho}$ is a $1$-cocycle; in particular $\cochain g_k^{\ho}\notin B^1(\K,\partial\K)$. Hence
$\eqclassspan{\mc G^{\ho}}\subset H^1(\K,\partial\K)$.
Moreover, since each $\cochain g_k^{\ho}\in B^1(\bar\K)$, the classes represented by $\mc G^{\ho}$ are trivial in $H^1(\bar\K)$, so $\eqclassspan{\mc G^{\ho}}\cap H^1(\bar\K)=\{0\}$.

Since each $\cochain g^{\ho}_k$ is associated with the $k$-th connected component $\partial\K_k$ of $\partial\K$, the elements of $\mc G^{\ho}$ are linearly independent.
As $\lvert \mc G^{\ho}\rvert = N^{\ho}-1$, applying \cref{lem:dimension.a} yields
\begin{equation*}
\rank{\eqclassspan{\mc G^{\ho}}}
= N^{\ho}-1
= \rank H^1(\K,\partial\K) - \rank H^1(\bar\K),
\end{equation*}
which establishes the decomposition stated in \cref{eq:decomposition.b}.
\end{proof}

\subsection{A third decomposition and proof of correctness of \cref{alg:contacts}}
We now discuss a preparatory result similarly to the one in \cref{lem:dimension.a}.
\begin{lemma}[Dimension computation via long exact sequences]
\label{lem:dimension.b}
Let $\K$ a $2$-dimensional CW complex with a non-empty boundary $\partial \K$.
Recalling that $\partial \K^{c} \coloneqq \partial \K \setminus \bigcup_{j=1}^{N^{\co}} \partial \K^{\co}_j$, where each $\partial \K^{\co}_j$ is a connected component of $\partial \K^{\co}$, the following expression relates the dimensions of $H^1(\K, \partial \K^{c})$ and $H^1(\K, \partial \K)$:
\begin{equation}
\rank{H^1(\K, \partial \K^{c})} - \rank{H^1(\K, \partial \K)} =
\begin{dcases} 
N^{\co} - 1 &\text{if $\K$ is orientable}, \\
N^{\co} &\text{otherwise}.
\end{dcases}
\label{eq:dimension.b}
\end{equation}
\end{lemma}
\begin{proof}
We proceed in a similar manner as in the proof of \cref{lem:dimension.a}, by expressing the dimensions of $H^1(\K, \partial \K^{c})$ and $H^1(\K, \partial \K)$ in terms of $H^1(\K)$.

First, let us consider the long exact sequence in relative cohomology \cite[Theorem 12]{spivak1979comprehensive} for the pair $(\K, \partial \K^{c})$:
\begin{align}
\begin{split}
0 &\longrightarrow \underset{0}{\underbrace{H^0(\K, \partial \K^{c})}} \xrightarrow{} \underset{1}{\underbrace{H^0(\K)}} \xrightarrow{} \underset{N^{c}}{\underbrace{H^0(\partial \K^{c})}} \rightarrow\\
&\xrightarrow{} H^1( \K, \partial \K^{c}) \xrightarrow{} H^1(\K) \xrightarrow{} H^1(\partial \K^{c}) \rightarrow\\
&\xrightarrow{} \underset{0}{\underbrace{H^2(\K, \partial \K^{c})}} \xrightarrow{} \underset{0}{\underbrace{H^2(\K)}} \xrightarrow{} \underset{0}{\underbrace{H^2(\partial \K^{c})}} \longrightarrow 0,
\label{eq:sequence.b.1}
\end{split}
\end{align}
where $N^{c}$ represents the number of connected components of $\partial \K^{c}$.

Next, let us note the following two facts:
\begin{enumerate}
\item In the exact sequence for $(\K,\partial \K)$ given in \cref{eq:sequence.a.2} of \cref{lem:dimension.a}, the dimension of $H^2(\K, \partial \K)$ is $1$ if $\K$ is orientable and $0$ if $\K$ is non-orientable.
\item In the exact sequence for $(\K, \partial \K^{c})$ in \cref{eq:sequence.b.1}, the dimension of $H^1(\partial \K^{c})$ satisfies the formula
\begin{equation}
\rank{H^1(\partial \K^{c})} = N^{c} - N^{\co},
\label{eq:passage}
\end{equation}
\end{enumerate}
which will now we prove.

Consider the set of all connected components $\partial \K_k$ of $\partial \K$, where $k \in \{1, \dots, N^{\ho}\}$.
After potentially rearranging the indices, we assume that there is a number $p \in \{1, \dots, N^{\ho}\}$ such that the first $p$ connected components, $\{\partial \K_1, \dots, \partial \K_p \}$, contain some contact regions $\partial \K^{\co}_j$ for $j \in \{1, \dots, N^{\co}\}$, while the remaining components $\{\partial \K_{p+1}, \dots, \partial \K_{N^{\ho}}\}$ do not.
If $p=N^{\ho}$, then the second set is empty.

For each $k \in \{1, \dots, N^{\ho}\}$, define $J^{\co}_k$ as the set of indices $j \in \{1, \dots, N^{\co}\}$ such that $\partial \K^{\co}_j \subset \partial \K_k$.
Let $|J^{\co}_k|$ denote the cardinality of $J^{\co}_k$, and define $N^{c}_k$ as the number of connected components of the set $\partial \K_k \setminus \bigcup_{j \in J^{\co}_k} \partial \K^{\co}_j$.

It is easy to check that $|J^{\co}_k| = N^{c}_k$ for all $k \in \{1, \dots, p\}$, and that $|J^{\co}_k|
= 0$ and $N^{c}_k =1$ for all $k \in \{p+1, \dots, N^{\ho}\}$.
Hence, we have the relation \[N^{\ho} - p = \sum_{k=p+1}^{N^{\ho}} N^{c}_k.\]

Also, since $\rank{H^1(\partial \K^{c})} = N^{\ho} -p$, we obtain the formula:
\begin{equation*}
N^{c}
=\sum_{k=1}^{N^{\ho}} N^{c}_k
=\sum_{k=1}^{p} N^{c}_k + \sum_{k=p+1}^{N^{\ho}} N^{c}_k 
= N^{\co} + \rank{H^1(\partial \K^{c})},
\end{equation*}
which is the desired result \cref{eq:passage}.

Finally, applying the Euler characteristic formula \cite[Theorem 2.44]{hatcher} to the exact sequences \cref{eq:sequence.b.1} and \cref{eq:sequence.a.2} in \cref{lem:dimension.a}, and using Facts 1 and 2, we derive the formula \cref{eq:dimension.b}.
\end{proof}

We are now ready to state the third decomposition of global loops, which also demonstrates the correctness of \cref{alg:contacts}.
\begin{lemma}[Third decomposition of global loops]
\label{lem:decomposition.c}
Let $\mc G^{\co}=\{\cochain g^{\co}_j\}_{j=1}^{\abs{\mc G^{\co}}}$ be the set of $1$-cochains computed by \cref{alg:contacts}, and define $\eqclassspan{\mc G^{\co}} \coloneqq \{\sum_{j=1}^{\abs{\mc G^{\co}}} \alpha_j [\cochain g^{\co}_j] \mid \alpha_j \in \R\}$ as the linear span of the equivalence classes associated with each generator $\cochain g^{\co}_j \in \mc G^{\co}$.
Then, the following group isomorphism holds
\begin{align}
H^1(\K, \partial \K^{c}) \cong H^1(\K, \partial \K) \oplus \eqclassspan{\mc G^{\co}}.
\label{eq:decomposition.c}
\end{align}
\end{lemma}
\begin{proof}
Let $E_M^{\mathrm{II}}$ be the set defined in \cref{alg:handles} at line 14, and used in \cref{alg:contacts} at line 9.
We distinguish between the two cases $\abs{E_M^{\mathrm{II}}} = 0$ and $\abs{E_M^{\mathrm{II}}} > 0$.

Assume $\abs{E_M^{\mathrm{II}}} = 0$. Then, according to \cref{alg:contacts}, each $1$-cochain $\cochain g^{\co}_j\in \mc G^{\co}$ is produced by \cref{alg:cocycle} from input pairs $(e,e')$ where $e,e'\in\partial\K$.
Based on \cref{eq:transformation} and \cref{eq:cocycle.face} in \cref{lem:cocycle}, each $\cochain g^{\co}_j$ is a $1$-cocycle that is nonzero on exactly one $1$-cell of the boundary component $\partial\K^{c}_j \subset \partial\K$, a feature that we refer to as the \emph{single-cell support property} for the remainder of the proof.

In particular, no $\cochain g^{\co}_j$ belongs to $B^1(\K,\partial\K^{c})$. Indeed, any $1$-cocycle in $B^1(\K,\partial\K^{c})$ must be nonzero on more than one $1$-cell of $\partial\K^{c}_j$, which would contradict the previous single-cell support property.
Hence, we obtain the inclusion $\eqclassspan{\mc G^{\co}}\subset H^1(\K,\partial\K^{c})$.
Moreover, the single-cell support property implies $\eqclassspan{\mc G^{\co}}\cap H^1(\K,\partial\K)=\{0\}$.

Since the elements of $\mc G^{\co}$ have disjoint, single-cell supports, they are linearly independent.
Using Point~(i) of \cref{lem:homology}, we deduce that $\bar\K$, and consequently $\K$, is orientable ($\bar\K$ is obtained by gluing disks along $\partial\K$; see \cite{gallier2013guide} for a discussion).
As $\abs{\mc G^{\co}} = N^{\co}-1$, applying \cref{lem:dimension.b} then yields
\begin{equation*}
\rank{\eqclassspan{\mc G^{\co}}} = N^{\co} - 1 = \rank H^1(\K,\partial \K^{c}) - \rank H^1(\K, \partial \K),
\end{equation*}
which establishes the decomposition in \cref{eq:decomposition.c}.

On the other hand, if $\abs{E_M^{\mathrm {II}}} > 0$, we consider the $1$-cochains $\cochain g_e^{(1)}$ and $\cochain g_e^{(2)}$ defined in \cref{alg:contacts} at lines 17 and 18.
Applying the same reasoning as in the proof of Point (ii) of \cref{lem:equivalence}, using the paths $\tilde p_{e,e^*}^{(1)}$ and $\dual p_{e,e^*}^{(2)}$ defined at lines 15 and 16, we deduce that both $1$-cochains $\cochain g_e^{(1)}$ and $\cochain g_e^{(2)}$ attain the same value on the $1$-cell $e^*$, chosen at line 13.
This implies that the $1$-cochain $\mathbf g^{\co}_{N^{\co}}$ defined at line 19, is a $1$-cocycle.
Using Point (ii) of \cref{lem:homology}, we deduce that in this case, $\bar \K$, and hence $\K$, must be non-orientable.
Repeating the same reasoning as above for the case $\abs{E_M^{\mathrm {II}}} = 0$ and applying \cref{lem:dimension.b}, we obtain the desired conclusion.
\end{proof}

\subsection{Final step}
To conclude, it is sufficient to combine Lemmas \ref{lem:decomposition.a}, \ref{lem:decomposition.b}, and \ref{lem:decomposition.c}.
\begin{lemma}[Correctness of \cref{alg:global.loops}]
Let $\mc G^{\ha}$, $\mc G^{\ho}$, $\mc G^{\co}$ denote the sets of global loops computed by Algorithms \ref{alg:handles}, \ref{alg:holes}, and \ref{alg:contacts}, respectively.
Then, we have the following isomorphism
\[
H^1(\K, \partial \K^{c}) \cong \eqclassspan{\mc G^{\ha}} \oplus \eqclassspan{\mc G^{\ho}} \oplus \eqclassspan{\mc G^{\co}}.
\]
\end{lemma}

\section{Application to BEM discretization of the EFIE}
\label{sec:EFIE}
In the previous sections, we established a rigorous algorithmic framework for computing a set of cohomology generators of $H^1(\K, \partial \K^c)$ for a generic CW complex.
In this section, we place this mathematical result in the context of the computational electromagnetics application mentioned in the introduction.
Specifically, we show how these cohomology generators, referred to in engineering literature as \emph{global loops}, appear naturally in the discretization of the EFIE for conducting surfaces with contacts.
\subsection{EFIE modeling of screens with contacts}
Let $\Gamma \subset \mathbb{R}^3$ be a compact surface representing the computational domain of a thin conducting body (a \emph{screen}).
We assume $\Gamma$ is a $2$-manifold with boundary, satisfying the Lipschitz regularity conditions detailed in \cite[Section 1.1]{buffa2003electric}.
Note that if the physical model includes $2$-dimensional contact regions (patches), $\Gamma$ is defined as the surface obtained after removing the interior of these patches (see, e.g., \cite[Section~II]{miano2005surface}).

An induced surface current density $\bm{J}$ flows on $\Gamma$, driven either by an incident electromagnetic field or by external voltage sources.
To model these connections, the boundary $\partial \Gamma$ is partitioned into two disjoint subsets as
\begin{equation}
\partial \Gamma = \partial \Gamma^{c} \cup \partial \Gamma^{\co}.
\end{equation}
Here, $\partial \Gamma^{\co}$ represents the union of $N^{\co}$ disjoint connected components (the \emph{contacts} or \emph{ports}).
These components may correspond either to the boundaries of removed patches (interior ports) or to specific subsets of the original boundary (edge ports).
The remaining part, $\partial \Gamma^{c}$, represents the insulating boundary.
See \cref{fig:proto} for an illustration.
\begin{figure}[H]
\includegraphics[width=0.45\textwidth,center]{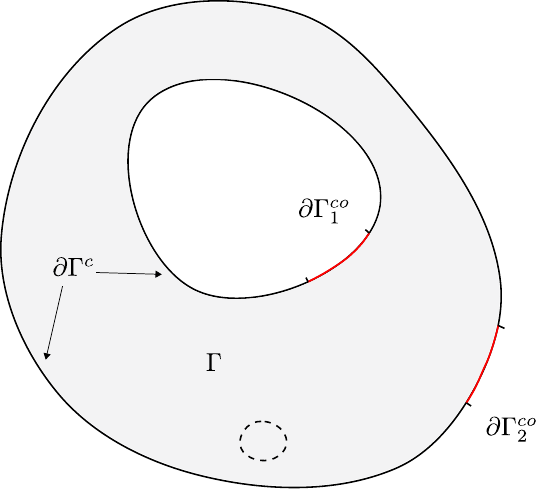}
\caption{
A screen $\Gamma$ whose boundary $\partial \Gamma$ decomposes into two connected components.
In red, the (open) subset $\partial \Gamma^{co}$ of $\partial \Gamma$ represents the contact region, and it decomposes into two connected components.
The dashed line indicates an example of a possible interior port,
whose associated patch must then be removed from $\Gamma$.
}
\label{fig:proto}
\end{figure}

The physical boundary conditions for the current density $\bm{J}$ differ on these two subsets.
Specifically, on the insulating boundary $\partial \Gamma^{c}$, no current can flow out of the surface, imposing the condition
\begin{equation}
\bm J \cdot \bm n_{\partial \Gamma} = 0 \quad \text{on} \quad \partial \Gamma^{c},
\label{eq:bc.partial}
\end{equation}
where $\bm n_{\partial \Gamma}$ is the unit outer normal tangent to $\Gamma$.
Conversely, on the contact interfaces $\partial \Gamma^{\co}$, current is allowed to flow in and out to connect with the external circuit.

\subsection{Formal definition of global loops}
To solve the EFIE numerically, we introduce a triangulation $\K$ covering the computational domain $\Gamma$.
The standard Galerkin BEM seeks the solution $\bm{J}$ in the finite element space $W^1_h$ spanned by Raviart-Thomas basis functions, subject to the boundary condition \cref{eq:bc.partial} on $\partial \K^c$.

The connection to the cochain framework is established via the \emph{de Rham map} \cite{arnold2018finite,dodziuk1976finite}.
This map provides a structural isomorphism between the finite element space $W^1_h$ and the space of relative $1$-cochains $C^1(\K, \partial \K^c; \mathbb{R})$.
Under this mapping:
\begin{itemize}
    \item The action of the surface divergence operator $\div_\Gamma$ corresponds to that of the coboundary operator $\delta^1$.
\item The subspace of divergence-free currents $\{\bm J \in W^1_h \mid \div_\Gamma \bm J = 0\}$ is isomorphic to the space of relative $1$-cocycles $Z^1(\K, \partial \K^c; \mathbb{R})$.
\item The subspace $\mathrm{Im} \, \curl_\Gamma$ is isomorphic to $B^1(\K, \partial \K^c; \mathbb{R})$.
\end{itemize}
Therefore, the space of solenoidal vector fields $\bm g$ required for the decomposition in \cref{eq:quasi-Helmholtz} is exactly isomorphic to the first relative cohomology group $H^1(\K, \partial \K^c; \R)$.

This physical reasoning leads directly to the formal definition used throughout this paper.
\begin{definition}[Definition of global loops]
\label{def:global.loops}
The set $\mc G = \{\cochain g_i \}_{i=1}^{\beta_1}$ of \emph{global loops} on the screen $\Gamma$ is defined as a set of representatives for a basis of the first relative cohomology group $H^1(\K, \partial \K^c; \mathbb{R})$.
Here, $\K$ is the triangulation of the computational domain $\Gamma$ (excluding interior ports), and $\partial \K^c$ is the subcomplex of the boundary where the insulation condition \cref{eq:bc.partial} holds.
\end{definition}
This definition clarifies that finding ``global loops'' is not merely a heuristic geometric search, but a rigorous algebraic task: identifying generators for the relative cohomology space $H^1(\K, \partial \K^c; \R)$.
The algorithm proposed in \cref{sec:algorithms} is specifically designed to 
address this task efficiently, addressing the complex topological configurations that arise in real-world engineering applications.
\section{Conclusions}

In this final section, we provide an overview of the algorithm's key features, its impact on engineering applications, and potential extensions that could guide future work.

We first emphasize that the algorithm's foundation in Discrete Morse Theory ensures its broad applicability.
The algorithms presented here are effective as long as an \emph{optimal discrete Morse function} can be computed.
For example, such functions can be determined even for pathological cases, such as surfaces connected by a single vertex or edge \cite{Lewiner2003OptimalDM}.
Additionally, these ideas can extend to handle \emph{non-manifold surfaces} (e.g., screens with junction points, as discussed in \cite{claeys2013integral,claeys2016integral}).
We propose that slight modifications to our algorithms could extend their applicability to these cases while maintaining linear worst-case complexity.
For instance, identifying manifold parts can be achieved in linear time by analyzing the coboundary elements of each edge.
By orienting a random polygon in each manifold part and iteratively orienting others using a \emph{breadth-first search} \cite{cormen2022introduction}, we expect the algorithm's linear time complexity to remain intact.

Another consideration is that the algorithm does not aim to produce a minimal or shortest set of cohomology generators.
This means that the generated set of cohomology generators may not have the most compact support (e.g., minimal length or the smallest number of edges).
For further insights, readers may consult Erickson et al. \cite{erickson2005greedy} for homology and D\l{}otko et al.
\cite{dlotko2014lazy} for cohomology, which are particularly relevant to our applications.
Achieving a truly minimal basis is often computationally prohibitive, necessitating a trade-off between runtime and the quality of the basis obtained.

Regarding the impact on engineering applications, the first aspect to highlight is the algorithm's ability to automatically compute global loops for surfaces, both with and without boundaries, and for orientable as well as non-orientable surfaces.
The primary algorithm, presented in \cref{sec:main.algo}, utilizes a three-step procedure to partition global loops into distinct disjoint classes.
Informally, these classes correspond to the presence of ``handles,'' ``holes,'' or contact regions.
The algorithm handles all these cases automatically, adapting to each situation to produce the correct output.
This is a crucial feature for engineering applications where manually verifying the explicit form and count of global loops is challenging.
Furthermore, this reliability is underpinned by the robust theoretical analysis provided earlier in the paper.

For example, consider the case where $\Gamma$ is a triangulated M\"obius strip. Our algorithm outputs no global loop, understood as a generator of $ H^1(\Gamma, \partial \Gamma) $. This result is easily validated theoretically via \cref{lem:decomposition.b}. Since the M\"obius strip has only one connected component, $ N^{\ho} = 1 $, it follows that $ H^1(\K, \partial \K) \cong H^1(\bar \K) $.
However, $ \bar \K $, obtained by gluing a disk to the M\"obius strip, is homeomorphic to a \emph{projective plane} (cf. \cite[pg.\ 19]{gallier2013guide}), and from the homological characterization of the projective plane, we know that $ H^1(\bar \K) \cong 0 $.

This observation corrects an error in \cite{hofmann2023low}, where it is claimed that the M\"obius strip supports a global loop (see the caption of Fig. 2 in \cite{hofmann2023low}).
This figure refers to their proposed procedure for computing global loops of the surface under consideration.
To avoid such pitfalls and ensure provably correct results, the algorithm proposed in this work is therefore explicitly required.

The second crucial practical aspect of the algorithm is its linear time worst-case complexity.
Specifically, the algorithms described in \cref{sec:algorithms} require only a list-based representation of a triangulation $\K$ of the surface $\Gamma$.
In this representation, each cell in $\K$ has an associated data structure that stores its neighboring cells.
In practical applications, it is generally assumed that the number of neighboring cells for each cell in $\K$ is bounded by a constant, reflecting common mesh regularity assumptions in boundary element methods.
Considering this constant, the use of standard graph algorithms for spanning trees results in a worst-case linear time complexity \cite{cormen2022introduction}.
Additionally, all algorithms in Section \ref{sec:algorithms} are entirely graph-based, which makes them straightforward to implement without requiring external libraries or specialized procedures.

\section*{Acknowledgements}
The author acknowledges the funding provided by the European Union through the MSCA EffECT, project number 101146324.

{\footnotesize
\bibliography{references}

\begin{thebibliography}{10}

\bibitem{buffa2003electric}
A.~Buffa and S.~H. Christiansen, ``The electric field integral equation on
  lipschitz screens: definitions and numerical approximation,'' {\em Numerische
  Mathematik}, vol.~94, pp.~229--267, 2003.

\bibitem{adrian2021electromagnetic}
S.~B. Adrian, A.~Dely, D.~Consoli, A.~Merlini, and F.~P. Andriulli,
  ``Electromagnetic integral equations: Insights in conditioning and
  preconditioning,'' {\em IEEE Open Journal of Antennas and Propagation},
  vol.~2, pp.~1143--1174, 2021.

\bibitem{vecchi1999loop}
G.~Vecchi, ``Loop-star decomposition of basis functions in the discretization
  of the efie,'' {\em IEEE Transactions on Antennas and Propagation}, vol.~47,
  no.~2, pp.~339--346, 1999.

\bibitem{andriulli2012loop}
F.~P. Andriulli, ``Loop-star and loop-tree decompositions: Analysis and
  efficient algorithms,'' {\em IEEE Transactions on Antennas and Propagation},
  vol.~60, no.~5, pp.~2347--2356, 2012.

\bibitem{andriulli2012well}
F.~P. Andriulli, K.~Cools, I.~Bogaert, and E.~Michielssen, ``On a
  well-conditioned electric field integral operator for multiply connected
  geometries,'' {\em IEEE transactions on antennas and propagation}, vol.~61,
  no.~4, pp.~2077--2087, 2012.

\bibitem{hofmann2023low}
B.~Hofmann, T.~F. Eibert, F.~P. Andriulli, and S.~B. Adrian, ``A low-frequency
  stable, excitation agnostic discretization of the right-hand side for the
  electric field integral equation on multiply-connected geometries,'' {\em
  IEEE Transactions on Antennas and Propagation}, vol.~71, no.~12,
  pp.~9277--9288, 2023.

\bibitem{hiptmair2002generators}
R.~Hiptmair and J.~Ostrowski, ``Generators of {$H_1(\Gamma_h,\mathbb Z)$} for
  triangulated surfaces: Construction and classification,'' {\em SIAM Journal
  on Computing}, vol.~31, no.~5, pp.~1405--1423, 2002.

\bibitem{Eppstein2002DynamicGO}
D.~Eppstein, ``Dynamic generators of topologically embedded graphs,'' {\em
  ArXiv}, vol.~cs.DS/0207082, 2002.

\bibitem{lazarus2001computing}
F.~Lazarus, M.~Pocchiola, G.~Vegter, and A.~Verroust, ``Computing a canonical
  polygonal schema of an orientable triangulated surface,'' in {\em Proceedings
  of the seventeenth annual symposium on Computational geometry}, pp.~80--89,
  2001.

\bibitem{dey2013efficient}
T.~K. Dey, F.~Fan, and Y.~Wang, ``An efficient computation of handle and tunnel
  loops via reeb graphs,'' {\em ACM Transactions on Graphics (TOG)}, vol.~32,
  no.~4, pp.~1--10, 2013.

\bibitem{gross2004electromagnetic}
P.~W. Gross and P.~R. Kotiuga, {\em Electromagnetic theory and computation: a
  topological approach}, vol.~48.
\newblock Cambridge University Press, 2004.

\bibitem{dlotko2012fast}
P.~D{\l}otko, ``A fast algorithm to compute cohomology group generators of
  orientable 2-manifolds,'' {\em Pattern Recognition Letters}, vol.~33, no.~11,
  pp.~1468--1476, 2012.

\bibitem{dlotko2013physics}
P.~D{\l}otko and R.~Specogna, ``Physics inspired algorithms for (co) homology
  computations of three-dimensional combinatorial manifolds with boundary,''
  {\em Computer Physics Communications}, vol.~184, no.~10, pp.~2257--2266,
  2013.

\bibitem{dlotko2014lazy}
P.~D{\l}otko and R.~Specogna, ``Lazy cohomology generators: A breakthrough in
  (co) homology computations for cem,'' {\em IEEE transactions on magnetics},
  vol.~50, no.~2, pp.~577--580, 2014.

\bibitem{greenberg1981algebraic}
M.~J. Greenberg and J.~R. Harper, {\em Algebraic Topology: A First Course}.
\newblock Boca Raton, FL: CRC Press, 1981.
\newblock Mathematics Lecture Note Series, Volume 58.

\bibitem{gallier2013guide}
J.~H. Gallier, D.~Xu, {\em et~al.}, {\em A guide to the classification theorem
  for compact surfaces}.
\newblock Springer, 2013.

\bibitem{wilton1981improving}
D.~Wilton and A.~Glisson, ``On improving the stability of the electric field
  integral equation at low frequency,'' in {\em Proc. IEEE Antennas and
  Propagation Soc. National Symp}, pp.~124--133, 1981.

\bibitem{arnold2018finite}
D.~N. Arnold, {\em Finite element exterior calculus}.
\newblock SIAM, 2018.

\bibitem{miano2005surface}
G.~Miano and F.~Villone, ``A surface integral formulation of maxwell equations
  for topologically complex conducting domains,'' {\em IEEE transactions on
  antennas and propagation}, vol.~53, no.~12, pp.~4001--4014, 2005.

\bibitem{wang2004generalized}
Y.~Wang, D.~Gope, V.~Jandhyala, and C.-J. Shi, ``Generalized kirchoff's current
  and voltage law formulation for coupled circuit-electromagnetic simulation
  with surface integral equations,'' {\em IEEE transactions on microwave theory
  and techniques}, vol.~52, no.~7, pp.~1673--1682, 2004.

\bibitem{suuriniemi2004state}
S.~Suuriniemi, J.~Kangas, L.~Kettunen, and T.~Tarhasaari, ``State variables for
  coupled circuit-field problems,'' {\em IEEE transactions on magnetics},
  vol.~40, no.~2, pp.~949--952, 2004.

\bibitem{suuriniemi2007driving}
S.~Suuriniemi, J.~Kangas, and L.~Kettunen, ``Driving a coupled field-circuit
  problem,'' {\em COMPEL-The international journal for computation and
  mathematics in electrical and electronic engineering}, vol.~26, no.~3,
  pp.~899--909, 2007.

\bibitem{hiptmair2021electromagnetic}
R.~Hiptmair and J.~Ostrowski, ``Electromagnetic port boundary conditions:
  Topological and variational perspective,'' {\em International Journal of
  Numerical Modelling: Electronic Networks, Devices and Fields}, vol.~34,
  no.~3, p.~e2839, 2021.

\bibitem{whitehead1949combinatorial}
J.~H. Whitehead, ``Combinatorial homotopy {I},'' {\em Bull. Amer. Math. Soc},
  vol.~55, no.~3, pp.~213--245, 1949.

\bibitem{cooke2015homology}
G.~E. Cooke and R.~L. Finney, {\em Homology of cell complexes}, vol.~2239.
\newblock Princeton University Press, 2015.

\bibitem{hatcher}
A.~Hatcher, ``Algebraic topology,'' 2001.

\bibitem{kaczynski2006computational}
T.~Kaczynski, K.~Mischaikow, and M.~Mrozek, {\em Computational homology},
  vol.~157.
\newblock Springer Science \& Business Media, 2006.

\bibitem{Kozlov2008}
D.~Kozlov, ``Combinatorial algebraic topology,'' in {\em Algorithms and
  computation in mathematics}, 2008.

\bibitem{pitassi2022inverting}
S.~Pitassi, R.~Ghiloni, and R.~Specogna, ``Inverting the discrete curl
  operator: a novel graph algorithm to find a vector potential of a given
  vector field,'' {\em Journal of Computational Physics}, vol.~466, p.~111404,
  2022.

\bibitem{spivak1979comprehensive}
M.~Spivak, ``A comprehensive introduction to differential geometry, publish or
  perish,'' {\em Inc., Berkeley}, vol.~2, 1979.

\bibitem{dodziuk1976finite}
J.~Dodziuk, ``Finite-difference approach to the hodge theory of harmonic
  forms,'' {\em American Journal of Mathematics}, vol.~98, no.~1, pp.~79--104,
  1976.

\bibitem{Lewiner2003OptimalDM}
T.~Lewiner, H.~Lopes, and G.~Tavares, ``Optimal discrete morse functions for
  2-manifolds,'' {\em Comput. Geom.}, vol.~26, pp.~221--233, 2003.

\bibitem{claeys2013integral}
X.~Claeys and R.~Hiptmair, ``Integral equations on multi-screens,'' {\em
  Integral equations and operator theory}, vol.~77, no.~2, pp.~167--197, 2013.

\bibitem{claeys2016integral}
X.~Claeys and R.~Hiptmair, ``Integral equations for electromagnetic scattering
  at multi-screens,'' {\em Integral Equations and Operator Theory}, vol.~84,
  pp.~33--68, 2016.

\bibitem{cormen2022introduction}
T.~H. Cormen, C.~E. Leiserson, R.~L. Rivest, and C.~Stein, {\em Introduction to
  algorithms}.
\newblock MIT press, 2022.

\bibitem{erickson2005greedy}
J.~Erickson and K.~Whittlesey, ``Greedy optimal homotopy and homology
  generators,'' in {\em SODA}, vol.~5, pp.~1038--1046, 2005.

\end{thebibliography}
}
\bibliographystyle{ieeetr}

\end{document}